\documentclass[12pt]{amsart}

\usepackage{a4wide}
\usepackage{amscd}

\usepackage{graphicx}

\usepackage{amsmath,amssymb,amsthm}

\usepackage{enumerate}
\usepackage[all]{xy}
\usepackage{tikz}
\usetikzlibrary{arrows}
\numberwithin{equation}{section}

\newtheorem{thm}{Theorem}[section]
\newtheorem{theorem}[thm]{Theorem}
\newtheorem{lemma}[thm]{Lemma}
\newtheorem{corollary}[thm]{Corollary}
\newtheorem{proposition}[thm]{Proposition}

\newtheorem*{MainTheorem}{Main Theorem}
\theoremstyle{definition}

\newtheorem{remark}[thm]{Remark}

\newtheorem{defn-thm}[thm]{Definition-Theorem}

\newcommand{\C}{\ensuremath{\mathbb{C}}}
\newcommand{\R}{\ensuremath{\mathbb{R}}}
\newcommand{\OO}{\ensuremath{\mathbb{O}}}
\newcommand{\FF}{\ensuremath{\mathbb{F}}}
\newcommand{\EE}{\ensuremath{\mathbb{E}}}
\newcommand{\HH}{\ensuremath{\mathbb{H}}}
\newcommand{\g}[1]{\ensuremath{\mathfrak{#1}}}

\newcommand{\Ss}{\ensuremath{\mathcal{S}}}

\DeclareMathOperator{\codim}{codim}

\DeclareMathOperator{\id}{id}

\DeclareMathOperator{\ad}{ad}

\begin{document}

\title[Submanifolds with constant principal curvatures]{Submanifolds with 
constant principal curvatures in symmetric spaces}

\author[J\"urgen Berndt and V\'ictor Sanmart\'in-L\'opez]{J\"urgen Berndt and V\'ictor Sanmart\'in-L\'opez}

\address{Department of Mathematics, King's College London, United Kingdom}
\email{jurgen.berndt@kcl.ac.uk}
\address{Departamento de Matem\'{a}tica Aplicada a las TIC, Universidad Polit\'{e}cnica de Madrid, Spain}
\email{victor.sanmartin@upm.es}

\thanks{The second author has been supported by projects PID2019-105138GB-C21 (AEI/FEDER, Spain) and ED431C 2019/10, ED431F 
2020/04 (Xunta de Galicia, Spain).}

\begin{abstract}
We study submanifolds whose principal curvatures, counted with multiplicities, do not depend on the normal direction. Such submanifolds, which we briefly call CPC submanifolds, are always austere, hence minimal, and have constant principal curvatures. Well-known classes of examples include totally geodesic submanifolds, homogeneous austere hypersurfaces, and singular orbits of cohomogeneity one actions. The main purpose of this article is to present a systematic approach to the construction and classification of homogeneous submanifolds whose principal curvatures are independent of the normal direction in irreducible Riemannian symmetric spaces of non-compact type and rank $\geq 2$. In particular, we provide a large number of new examples of non-totally geodesic CPC submanifolds not coming from cohomogeneity one actions (note that only one example was known previously, namely a particular 11-dimensional submanifold of the Cayley hyperbolic plane).
\end{abstract}

\maketitle

\section{Introduction} \label{intro}
Consider a cohomogeneity one action on a connected complete Riemannian manifold $M$. If there is a singular orbit $P$ of this action, then the principal curvatures of $P$ do not depend on the normal directions. More precisely, if $\xi_1$ and $\xi_2$ are two unit normal vectors of $P$, then the principal curvatures of $P$ with respect to $\xi_1$ and $\xi_2$ are the same, counted with multiplicities. This is a simple consequence of the homogeneity of $P$ and the fact that the slice representation of the action at any point $p \in P$ acts transitively on the unit sphere in the normal space of $P$ at $p$. 

An obvious consequence is that every singular orbit of a cohomogeneity one action is an austere, and hence minimal, submanifold. Austere submanifolds were introduced by Harvey and Lawson \cite{HL82} in the context of calibrated geometries. Another consequence is that every singular orbit of a cohomogeneity one action is a submanifold with constant principal curvatures. By definition, a submanifold $P$ of a Riemannian manifold $M$ has constant principal curvatures if the principal curvatures of $P$ are constant for any parallel normal vector field of $P$ along any piecewise differentiable curve in $P$. Submanifolds with constant principal curvatures were introduced and studied by Heintze, Olmos and Thorbergsson \cite{HOT91} in the context of isoparametric submanifolds. A remarkable result of their work states that a submanifold of a Euclidean space has constant principal curvatures if and only if it is an isoparametric submanifold or a focal manifold of an isoparametric submanifold. 

Assume that $M$ is a standard real space form, that is, $M$ is the real hyperbolic space $\R H^n$, the Euclidean space $\EE^n$, or the sphere $S^n$, with their standard metrics of constant curvature $-1,0,+1$ respectively. Let $P$ be a submanifold of $M$ with $\codim(P) \geq 2$. Using Jacobi 
field theory one can show that the principal curvatures of $P$ are independent of the normal direction if and only if the tubes (of sufficiently small radii) around $P$ have constant principal curvatures. According to \'Elie Cartan \cite{Ca38}, a hypersurface of a space of constant curvature 
has constant principal curvatures if and only if it is isoparametric. For 
$\R H^n$ and $\EE^n$, the classification problem for isoparametric hypersurfaces was solved by Cartan \cite{Ca38} and Segre \cite{Se38} approximately 80 years ago. In contrast, the problem for $S^n$ turned out to be very challenging and was solved only recently by Chi (\cite{Ch16}; see \cite{Ch17} for a survey).

One of the implications in the above characterization was recently generalized by Ge and Tang \cite{GeTang} to arbitrary Riemannian manifolds: Let 
$P$ be a submanifold of a Riemannian manifold $M$ with $\codim(P) \geq 2$ 
for which the tubes around it (for sufficiently small radii) are isoparametric hypersurfaces with constant principal curvatures. Then the principal curvatures of $P$ are independent of the normal direction. The converse 
is not true, since tubes around totally geodesic submanifolds are generally not isoparametric (see e.g.~\cite{DDS17}).

It is evident from the above that submanifolds whose principal curvatures 
are independent of the normal direction arise in various geometric contexts. However, there seems to be no systematic study in a more general setting. This is somewhat surprising, given that the condition on the principal curvatures is remarkably simple and natural. For the sake of brevity, we call such a submanifold a CPC submanifold (CPC = constant principal curvatures). Note that our notion is more restrictive than the one in \cite{HOT91}: Every CPC submanifold is a submanifold with constant principal 
curvatures in the sense of \cite{HOT91}. 

The purpose of this paper is to develop a systematic approach addressing existence and classification questions of CPC submanifolds in Riemannian symmetric spaces of non-compact type. Our focus here is on non-totally geodesic CPC submanifolds that are not orbits of cohomogeneity one actions, 
since the latter ones have already been investigated thoroughly by others 
(see e.g.\ \cite{BB01}, \cite{BT04}, \cite{BT07}, \cite{BT13}). Our study 
was triggered by the remarkable discovery in \cite{DD13} of an $11$-dimensional homogeneous CPC submanifold of the Cayley hyperbolic plane $\OO H^2$ that is not an orbit of a cohomogeneity one action. To our knowledge, this is the only known non-totally geodesic CPC submanifold in an irreducible Riemannian symmetric space of non-compact type that is not an orbit of a cohomogeneity one action.

Let $M = G/K$ be an irreducible Riemannian symmetric space of non-compact type, where $G = I^o(M)$ is the identity component of the isometry group of $M$ and $K$ is the isotropy group of $G$ at a point $o \in M$. Let $\g{g} = \g{k} \oplus \g{p}$ be the corresponding Cartan decomposition of the Lie algebra $\g{g}$ of $G$. Choose a maximal abelian subspace $\g{a}$ of $\g{p}$ and let $\g{g} = \g{g}_{0} \oplus \left( \bigoplus_{\alpha \in \Delta} \g{g}_{\alpha} \right)$ be the induced restricted root space decomposition of $\g{g}$, where $\Delta$ denotes the set of restricted roots. Let $\g{g} = \g{k} \oplus \g{a} \oplus \g{n}$ be the corresponding Iwasawa decomposition of $\g{g}$. Denote by $AN$ the solvable closed connected subgroup of $G$ with Lie algebra $\g{a} \oplus \g{n}$. Then $M$ is isometric to $AN$ endowed with a suitable left-invariant Riemannian 
metric. Let $\Pi$ be a set of simple roots for $\Delta$ and denote by $\Pi'$ the set of simple roots $\alpha \in \Pi$ with $2 \alpha \notin \Delta$. Note that there is at most one simple root in $\Pi$ that does not belong to $\Pi'$, and this happens precisely when the restricted root system of $G/K$ is of type $BC_r$. Denote by $\g{k}_{0} = \g{g}_{0} \cap \g{k}$ the principal isotropy subalgebra of $\g{k}$. If $U$ is a vector space with an inner product and $W \subseteq U$ is a linear subspace, we will denote by $U \ominus W$ the orthogonal complement of $W$ in $U$ with respect to the inner product. 

We now state the main result of this paper.

\begin{MainTheorem}
	Let $\g{s} = \g{a} \oplus (\g{n} \ominus V)$ be a subalgebra of $\g{a} 
\oplus \g{n}$ with $V \subseteq \bigoplus_{\alpha \in \Pi'} \g{g}_{\alpha}$. Let $S$ be the connected closed subgroup of $AN$ with Lie algebra $\g{s}$. Then the orbit $S \cdot o$ is a CPC submanifold of $M = G/K$ if and only if one of the following statements holds:
	\begin{itemize}
		\item[{\rm (I)}] There exists a simple root $\lambda \in \Pi'$ with $V \subset \g{g}_{\lambda}$. \label{main:simple:examples}
		\item[{\rm (II)}] There exist two non-orthogonal simple roots $\alpha_0,\alpha_1 \in \Pi'$ with $|\alpha_0| = |\alpha_1|$ and subspaces $V_0 \subseteq \g{g}_{\alpha_0}$ and $V_1 \subseteq \g{g}_{\alpha_1}$ such that 
$V = V_0 \oplus V_1$ and one of the following conditions holds: \label{main:new:examples}
		\begin{itemize}
			\item[{\rm (i)}] $V_0 \oplus V_1 = \g{g}_{\alpha_0} \oplus \g{g}_{\alpha_1}$; \label{singular:orbit} 
			\item[{\rm (ii)}] $V_0 \oplus V_1$ is a proper subset of $\g{g}_{\alpha_0} \oplus \g{g}_{\alpha_1}$ and
			\begin{itemize}
				\item[{\rm (a)}] $V_0$ and $V_1$ are isomorphic to $\mathbb{R}$; or
				\item[{\rm (b)}] $V_0$ and $V_1$ are isomorphic to $\mathbb{C}$ and there exists $T \in \g{k}_0$ such that $\ad(T)$ defines complex structures 
on $V_0$ and $V_1$ and vanishes on $[V_0, V_1]$; or \label{main:complex}
				\item[{\rm (c)}] $V_0$ and $V_1$ are isomorphic to $\mathbb{H}$ and there exists a subset $\g{l} \subseteq \g{k}_0$ such that $\ad(\g{l})$ defines quaternionic structures on $V_0$ and $V_1$ and vanishes on $[V_0, V_1]$. \label{main:quaternionic}
			\end{itemize}
		\end{itemize}
	\end{itemize}
	Moreover, only the submanifolds given by {\rm (I)} and {\rm (II)(i)} can 
appear as singular orbits of cohomogeneity one actions.
\end{MainTheorem} 

Using algebraic theory of semisimple Lie algebras, this result produces many examples of non-totally geodesic homogeneous CPC submanifolds that are not orbits of cohomogeneity one actions.  

The submanifolds in (I) can be thought of as canonical extensions of submanifolds in real hyperbolic spaces. According to \cite{Ca38}, all these examples are singular orbits of cohomogeneity one actions. Thus, from the above mentioned result by Ge and Tang, we obtain directly that their principal curvatures are independent of the normal direction. The interesting 
submanifolds are therefore the ones in (II). We will construct these submanifolds explicitly and compute their shape operator. For this purpose, we first generalise the important concept of strings generated by a single 
root to a more general concept of strings generated by two roots. This more general concept will then induce a natural decomposition of the tangent space of the submanifold into subspaces that are invariant under the shape operator. The root space structure will then allow us to calculate explicitly the shape operator when restricted to each of these invariant subspaces. This technique is original and we hope that it can be applied also in other situations. We will also construct explicitly the complex and 
quaternionic structures mentioned in the Main Theorem.

Roughly, our strategy for this paper is divided into three parts: a construction part, a classification part and a description part. We will first 
construct the submanifolds introduced in the Main Theorem and, in particular, we will see that all the cases occur. We will also prove that their principal curvatures are independent of the normal direction. We will then prove, in the classification part, that there are no other such submanifolds under the given hypotheses. Finally, in the description part, we will find out which of these submanifolds do not come from cohomogeneity one actions. 

The paper is organized as follows. In Section \ref{preliminaries}, we introduce the main tools used for our investigations. In Section \ref{construction:examples}, we start by introducing the general setting for constructing the new CPC submanifolds. We show that for understanding the principal curvatures of those submanifolds it suffices to determine a specific decomposition of the tangent space into certain subspaces that are invariant by the shape operator. Calculating the shape operator when restricted 
to such a subspace turns out to be equivalent to studying our problem for 
a symmetric space whose Dynkin diagram is of type $A_2$, that is, for the 
symmetric spaces $SL_{3}(\mathbb{R})/SO_{3}$, $SL_{3}(\mathbb{C})/SU_{3}$, $SL_{3}(\mathbb{H})/Sp_{3}$ and  $E^{-26}_6/F_4$. In the final part of the section we prove the construction and classification part of the Main 
Theorem for these particular symmetric spaces.  In Section~\ref{canonical:extension} we will show that all the submanifolds in the Main Theorem are indeed CPC submanifolds. Thus, in Section~\ref{canonical:extension}, we 
finish the construction part of the paper. Section~\ref{classification} is devoted to the classification part of the Main Theorem. Actually, we will see that if the subspace $V$ does not satisfy the conditions in (I) or 
(II), then $S \cdot o$ cannot be a CPC submanifold.  In Section \ref{new:examples}, we analyze if the new submanifolds can be realized as singular 
orbits of cohomogeneity one actions. Finally, in Section \ref{geometric explanations}, we provide some further geometric explanations of the examples in the rank $2$ cases.

\section{Preliminaries}\label{preliminaries}

Let $M$ be a connected Riemannian symmetric space of non-compact type. Denote by $G$ the identity component of the isometry group of $M$. Let $K$ be the isotropy group of $G$ at a point $o \in M$. Then $M$ is diffeomorphic to $G/K$. Let $\g{g}$ and $\g{k}$ be the real Lie algebras of $G$ and 
$K$, respectively. Then $\g{k}$ is a maximal compact subalgebra of the real semisimple Lie algebra $\g{g}$. Let $B$ be the Killing form of $\g{g}$ 
and define $\g{p}$ as the orthogonal complement of $\g{k}$ in $\g{g}$ with respect to $B$. Then $\g{g} = \g{k} \oplus \g{p}$ is a Cartan decomposition of the Lie algebra $\g{g}$. If we denote by $\theta$ the corresponding Cartan involution, then $\langle X, Y \rangle_{B_{\theta}} = -B(X, 
\theta Y)$ defines a $\theta$-invariant positive definite inner product on $\g{g}$. Note also that $\theta\rvert _{\g{k}} = \id_{\g{k}}$ and $\theta \rvert_{\g{p}} = -\id_{\g{p}}$. The map $\phi_o \colon G \to M,\ g 
\mapsto g(o)$, allows us to identify $\g{p}$ with the tangent space $T_{o}M$. Then, by a suitable normalization, the metric of $M$ at $T_oM$ coincides with the restriction of the inner product $\langle \cdot, \cdot \rangle_{B_{\theta}}$ to $\g{p} \times \g{p}$.

Let $\g{a}$ be a maximal abelian subspace of $\g{p}$ and denote by $\g{a}^{*}$ its dual vector space. For each $\lambda \in \g{a}^{*}$ define 
\[
\g{g}_{\lambda} = \{ X \in \g{g} : \ad(H) X = \lambda(H) X \ \text{for all} \ H \in \g{a}\}.
\]
If $\lambda \neq 0$ and $\g{g}_{\lambda} \neq \{0\}$, then $\lambda$ is called a restricted root and $\g{g}_{\lambda}$ is called a restricted root 
space. We denote by $\Delta$ the set of all restricted roots and put $\Delta_0 = \Delta \cup \{0\}$. The root spaces $\g{g}_{\lambda}$ and $\g{g}_{-\lambda}$ are related by $\theta \g{g}_{\lambda} = \g{g}_{-\lambda}$ for all $\lambda \in \Delta$. Note that $\{\ad(H) : H \in \g{a} \}$ is a family of pairwise commuting self-adjoint endomorphisms of $\g{g}$. This guarantees that $\Delta$ is a finite and non-empty subset of $\g{a}^{*}$. The $\langle \cdot, \cdot \rangle_{B_{\theta}}$-orthogonal decomposition
\[
\g{g}=\g{g}_0 \oplus \left( \bigoplus_{\lambda \in \Delta} \g{g}_\lambda \right)
\]
is called the restricted root space decomposition of $\g{g}$ with respect 
to $\g{a}$. Furthermore, the restricted root spaces satisfy the bracket relation
\begin{equation}\label{bracket:relation}
[\g{g}_{\lambda},\g{g}_{\mu}] \subseteq \g{g}_{\lambda+\mu}
\end{equation}
for all $\lambda,\mu \in \Delta$, where $\g{g}_{\lambda+\mu} = \{0\}$ if $\lambda+\mu \notin \Delta$. For each $\lambda \in \Delta$ the root vector $H_{\lambda} \in \g{a}$ is defined by $\lambda(H) = \langle H, H_{\lambda} \rangle_{B_{\theta}}$ for all $H \in \g{a}$. Thus we have an inner product in $\Delta$ given by $\langle \alpha, \lambda \rangle = \langle H_{\alpha}, H_{\lambda} \rangle_{B_{\theta}}$, for each $\alpha$, $\lambda \in \Delta$. Let $\Pi$ be a set of simple roots for $\Delta$ and denote by $\Delta^+$ the resulting set of positive roots. We also define
\[
\Pi^\prime = \{\alpha \in \Pi : 2\alpha \notin \Delta^+\}.
\]
Note that $\Pi^\prime = \Pi$ if and only if the restricted root system is not of type $BC_r$. If the root system is of type $BC_r$, then there exists exactly one simple root $\alpha \in \Pi$ with $2\alpha \in \Delta$. 
Each root $\lambda \in \Delta$ can be written as $\lambda = \sum_{\alpha \in \Pi }n_{\alpha} \alpha$, where the coefficients $n_{\alpha}$ are either all non-negative or all non-positive integers depending on whether $\lambda$ is a positive root or a negative root, respectively. For each root $\lambda \in \Delta^{+}$, the sum $l(\lambda) = \sum_{\alpha \in \Pi} n_{\alpha}$ is called the level of the root $\lambda$. 

Let $\alpha \in \Delta$ and $\lambda \in \Delta_0$. The $\alpha$-string containing $\lambda$ is defined as the set of all elements in $\Delta_0$ of the form $\lambda + n \alpha$ with $n \in \mathbb{Z}$.  We state some important facts about strings that will be used throughout this work.

\begin{proposition}\label{p2.48}
	\cite[Proposition 2.48]{K} Let $\Delta$ be the restricted root system of 
a Riemannian symmetric space of non-compact type.
	\begin{enumerate}[{\rm (i)}]
		\item If $\alpha \in\Delta$, then $-\alpha \in \Delta$. \label{oposite}
		\item If $\alpha \in \Delta$ and $\lambda \in \Delta_0$, then
		\[
		A_{\alpha,\lambda} = \frac{2 \langle \lambda, \alpha \rangle}{|\alpha|^2} \in \{0, \pm 1, \pm 2, \pm 3\, \pm 4\},
		\]
		and $\pm 4$ occurs only when $\Delta$ is non-reduced  and  $\lambda = 
\pm 2 \alpha$.
		\item If $\alpha,\lambda\in\Delta$ are non-proportional and $|\lambda| \leq |\alpha|$, then $A_{\alpha,\lambda} \in \{0, \pm 1\}$. \label{pr:short}
		\item Let $\alpha,\lambda \in \Delta$. If $\langle \alpha, \lambda \rangle > 0$, then $\alpha - \lambda \in \Delta_0$, and if $\langle \alpha, \lambda \rangle < 0$, then $\alpha + \lambda \in \Delta_0$. \label{sd}
		\item If $\alpha\in\Delta$ and $\lambda\in\Delta_0$, then the $\alpha$-string containing $\lambda$ has the form $\lambda + n \alpha$ for $-p \leq n \leq q$ with $p, q \geq 0$. There are no gaps. Furthermore $p-q = A_{\alpha,\lambda}$. The $\alpha$-string containing $\lambda$ contains at most four roots. \label{pr:string}
	\end{enumerate}
\end{proposition}

In the following result we clarify the relation between the dimensions of 
the root spaces involved in a string. 

\begin{lemma}\label{root:spaces:dimension}
	Let $\alpha,\lambda \in \Delta^+$ be linearly independent.
	\begin{enumerate}[{\rm (i)}]
		\item If the $\alpha$-string containing $\lambda$ consists of $\lambda,\lambda+\alpha$, then $A_{\alpha,\lambda} = -1$ and $\dim(\g{g}_{\lambda}) = \dim(\g{g}_{\lambda + \alpha})$.
		\item If the $\alpha$-string containing $\lambda$ consists of $\lambda,\lambda+\alpha,\lambda+2\alpha$, then $A_{\alpha,\lambda} = -2$, $\dim(\g{g}_{\alpha}) = \dim(\g{g}_{\lambda + \alpha })$ and $\dim(\g{g}_{\lambda}) = \dim(\g{g}_{\lambda + 2 \alpha})$.
	\end{enumerate}
\end{lemma} 

\begin{proof}
	The statements about $A_{\alpha,\lambda}$ follow from Proposition \ref{p2.48}(\ref{pr:string}). We denote by $s_{\alpha} (\lambda) = \lambda- A_{\alpha, \lambda} \alpha$ the Weyl reflection of $\alpha$.
	
	If the $\alpha$-string containing $\lambda$ is $\lambda,\lambda+\alpha$, 
then $A_{\alpha,\lambda} = -1$ and  $s_\alpha(\lambda) = \lambda - A_{\alpha,\lambda}\alpha = \lambda+\alpha$. Since the Weyl reflection $s_\alpha$ interchanges $\lambda$ and $\lambda+\alpha$, we get $\dim(\g{g}_{\lambda}) = \dim(\g{g}_{\lambda + \alpha})$.
	
	Next, assume that the $\alpha$-string containing $\lambda$ is $\lambda,\lambda+\alpha,\lambda+2\alpha$. Then we have $A_{\alpha,\lambda} = -2$ and $s_\alpha(\lambda) = \lambda - A_{\alpha,\lambda}\alpha = \lambda+2\alpha$, which implies $\dim(\g{g}_{\lambda}) = \dim(\g{g}_{\lambda + 
2 \alpha})$. 
	The only root systems of rank two with $\alpha$-strings of length $3$ and containing only positive roots are $B_2$ and $BC_2$. In the $B_2$-case there is only one such string $\lambda,\lambda+\alpha,\lambda+2\alpha$, namely when $\alpha,\lambda$ are the simple roots of $B_2$ and $\alpha$ is 
the shortest of the two roots. In this case we have $s_{\lambda}(\alpha) = \lambda+\alpha$, which implies $\dim(\g{g}_{\alpha}) = \dim(\g{g}_{\lambda + \alpha })$. In the $BC_2$-case there is another such string $\lambda^\prime,\lambda^\prime+\alpha^\prime,\lambda^\prime+2\alpha^\prime$ with $\lambda^\prime = 2\alpha$ and $\alpha^\prime = \lambda$. In this situation we have $s_{\lambda^\prime}(\alpha^\prime) = s_{2\alpha}(\lambda) = 2\alpha + \lambda = \lambda^\prime + \alpha^\prime$, which implies $\dim(\g{g}_{\alpha^\prime}) = \dim(\g{g}_{\lambda^\prime + \alpha^\prime})$.
\end{proof}

After these algebraic considerations about the structure of the root system $\Delta$, we will focus on the Riemannian structure of $M$. From the bracket relations (\ref{bracket:relation}) we easily see that $\g{n} = \bigoplus_{\lambda \in \Delta^{+}} \g{g}_{\lambda}$ is a nilpotent subalgebra of $\g{g}$. The direct sum decomposition $\g{g} = \g{k} \oplus \g{a} \oplus \g{n}$ is an Iwasawa decomposition of $\g{g}$. Since $[\g{a}, \g{n}] = \g{n}$, we deduce that $\g{a} \oplus \g{n}$ is a solvable subalgebra of $\g{g}$. Let $A$, $N$ and $AN$ be the connected closed subgroups of $G$ with Lie algebras $\g{a}$, $\g{n}$ and $\g{a} \oplus \g{n}$ respectively. Then $G$ is diffeomorphic to the product $KAN$. Furthermore, the solvable Lie group $AN$ acts simply transitively on $M$, which allows us to equip $AN$ with a left-invariant Riemannian metric $\langle \cdot, \cdot \rangle_{AN}$ so that $AN$ and $M$ are isometric Riemannian manifolds. 
As shown in \cite{Ta11}, this metric is given by 
\[
\langle H_1 + X_1, H_2 + X_2 \rangle_{AN} = \langle H_1, H_2 \rangle_{B_{\theta}} + \frac{1}{2}\langle X_1, X_2 \rangle_{B_{\theta}},
\] 
where $H_1,H_2 \in \g{a}$ and $X_1,X_2 \in \g{n}$. Let $\nabla$ be the Levi-Civita connection of $AN = M$. Using the identity 
\begin{equation}\label{cartan:inner}
\langle \ad(X) Y, Z \rangle_{B_{\theta}} = -\langle  Y, \ad(\theta X) Z 
  \rangle_{B_{\theta}}  
\end{equation}
and the Koszul formula, we deduce (see e.g.\ \cite{BDT10}) 
\begin{equation}\label{levi}
4 \langle \nabla_{X} Y, Z  \rangle_{AN} =  \langle [X, Y] + (1-\theta)[\theta X, Y], Z  \rangle_{B_{\theta}}.
\end{equation}

In this paper, we are interested in a particular class of submanifolds of 
$M$. Let $\g{s}$ be a subalgebra of $\g{a} \oplus \g{n}$ and $S$ the connected closed subgroup of $AN$ with Lie algebra $\g{s}$. We will study the 
orbit $S \cdot o$, which by definition is a homogeneous submanifold of $M$. We can identify the tangent space $T_o(S \cdot o)$ with $\g{s}$ and the normal space $\nu_o(S \cdot o)$ with the orthogonal complement $V$ of $\g{s}$ in $\g{a} \oplus \g{n}$. The shape operator $\Ss_{\xi}$ of $S \cdot o$ with respect to a unit normal vector $\xi \in V$ is given by
\begin{equation}\label{definition:shape:operator}
\Ss_{\xi} X = - \left( \nabla_{X} \xi \right)^{\top},
\end{equation}
where $X \in \g{s}$ and $(\cdot)^{\top}$ denotes the orthogonal projection onto the space $T_o(S \cdot o) \cong \g{s}$. 

In order to simplify some arguments in this paper, we state a result which will allow us to use the Levi-Civita connection more efficiently. 

\begin{lemma}\label{lemma:a,k0} 
	Let $\lambda \in \Delta^{+}$ and $X,Y \in \g{g}_{\lambda}$ be orthogonal. 
	\begin{enumerate}[{\rm (i)}]
		\item $[\theta X, X] = 2\langle X, X \rangle_{AN} H_{\lambda}=\langle X, X \rangle_{B_{\theta}} H_{\lambda}$. \label{X:X}
		\item $[\theta X, Y] \in \g{k}_0 = \g{g}_0 \ominus \g{a}$. \label{ak0}
		\item If $2 \lambda \notin \Delta^+$, then 
		\[ \langle [\theta X, Y], [\theta X, Z]  \rangle_{B_{\theta}} = 4 |\lambda|^{2} \langle X, X \rangle_{AN} \langle Y, Z \rangle_{AN} \] for all 
$Z \in \g{g}_{\lambda}$ orthogonal to $X$. \label{k0:elements}
		\item If $2 \lambda \notin \Delta^+$, then $\nabla_X Y = 0$. \label{sh:itself}
	\end{enumerate}
\end{lemma}

\begin{proof}
	Firstly, we have $\theta[\theta X, X] = -[\theta X, X]$. Using the bracket relation in (\ref{bracket:relation}), the Cartan decomposition $\g{g} = \g{k} \oplus \g{p}$ and the facts that $\theta\rvert _{\g{k}} = \id_{\g{k}}$ and $\theta \rvert_{\g{p}} = -\id_{\g{p}}$, we deduce that $[\theta X, X] \in \g{a} = \g{p} \cap \g{g}_0$. Now, using (\ref{cartan:inner}) and the definition of restricted root space, we obtain 
	\begin{align*}
	\langle [\theta X, X], H_{\lambda}  \rangle_{B_{\theta}}  = \langle X, 
[H_{\lambda}, X] \rangle_{B_{\theta}} = |\lambda|^2 \langle X, X \rangle_{B_{\theta}}  = 2 |\lambda|^2 \langle X, X \rangle_{AN}.
	\end{align*}  
	A similar calculation shows that $\langle [\theta X, X], H  \rangle_{B_{\theta}} = 0$ for all $H$ orthogonal to $H_{\lambda}$. Then, we get $[\theta X, X] = 2\langle X, X \rangle_{AN} H_{\lambda}$, which proves (\ref{X:X}). 
	
	For (\ref{ak0}), let $H \in \g{a}$. Clearly, $[\theta X, Y] \in \g{g}_{0}$ by (\ref{bracket:relation}). However, using again (\ref{cartan:inner}) 
and the definition of restricted root space, we obtain  $\langle [\theta X, Y], H \rangle_{B_{\theta}} = \lambda(H) \langle Y, X \rangle_{B_{\theta}} = 0$, which implies $[\theta X, Y] \in \g{k}_0 = \g{g}_{0} \ominus \g{a}$. 
	
	For (\ref{k0:elements}), let $Z \in \g{g}_{\lambda}$ be orthogonal to $X$. Then, using (\ref{cartan:inner}), the Jacobi identity, the assumption that $2 \lambda \notin \Delta^+$, (\ref{X:X}), and the definition of restricted root space, we have 
	\begin{align*}
	\langle [\theta X, Y], [\theta X, Z]  \rangle_{B_{\theta}} & = - \langle Y, [X,[\theta X, Z]] \rangle_{B_{\theta}} = \langle Y, [Z,[X, \theta 
X]] \rangle_{B_{\theta}} \\
	& = \langle Y, [[\theta X, X],Z] \rangle_{B_{\theta}} 
	= 2 |\lambda|^{2} \langle X, X \rangle_{AN} \langle Y, Z \rangle_{B_{\theta}} \\
	& = 4 |\lambda|^{2} \langle X, X \rangle_{AN} \langle Y, Z \rangle_{AN}. 
	\end{align*}
	
	In order to prove (\ref{sh:itself}), we will use equation (\ref{levi}) directly. On the one hand, from (\ref{bracket:relation}), we obtain that the vectors $[\theta X, Y]$ and $[X, \theta Y]$ both belong to $\g{g}_{0}$. Since $\g{n} = \bigoplus_{\lambda \in \Delta^{+}} \g{g}_{\lambda}$, 
we deduce that they have trivial projections onto $\g{n}$. From (\ref{ak0}) we conclude they have also trivial projections onto $\g{a}$ and consequently onto $\g{a} \oplus \g{n}$. On the other hand, the element $[X, Y]$ 
vanishes because of (\ref{bracket:relation}) and the assumption that $2 \lambda \notin \Delta^+$. Then, we deduce
	\begin{align*}
	4 \langle \nabla_{X} Y, Z  \rangle_{AN} & = \langle [X, Y] + [\theta X, Y] - [X, \theta Y], Z  \rangle_{B_{\theta}} = 0
	\end{align*}
	for all $Z \in \g{a} \oplus \g{n}$. This finishes the proof.
\end{proof}

The next result will be used later for calculating principal curvatures.

\begin{lemma} \label{lemma:ad}
	Let $\gamma \in \Delta^{+}$ be the root of minimum level in its non-trivial $\nu$-string, for $\nu \in \Delta^{+}$ non-proportional to $\gamma$. Let $X \in \g{g}_{\gamma}$ and $\xi \in \g{g}_{\nu}$ with $\langle \xi, \xi \rangle_{AN}  = 1$. 
	\begin{enumerate}[{\rm (i)}]
		\item $\ad(\xi)\rvert_{\g{g}_{\gamma}} \colon \g{g}_{\gamma} \to \g{g}_{\gamma + \nu}$ is an injective map preserving the inner product up to a positive constant.\label{lemma:ad:i}
		\item $[\theta \xi, [\xi, X]] = A_{\nu, \gamma} |\nu|^{2} X$.\label{lemma:ad:ii}
		\item $[\theta \xi, [\xi, [\xi, X]]] = (A_{\nu, \gamma+\nu} + A_{\nu, 
\gamma}) |\nu|^{2}[\xi, X]$.\label{lemma:ad:iii}
		\item If $A_{\nu, \gamma} \leq -2$, then 
		\[ [\theta \xi, [\xi, [\xi, [\xi, X]]]] = (A_{\nu, \gamma + 2\nu} + A_{\nu, \gamma+\nu} + A_{\nu, \gamma}) |\nu|^{2}[\xi,[\xi, X]]. \] \label{lemma:ad:iv}
	\end{enumerate}
\end{lemma}
\begin{proof}
	Since $\gamma$ is the root of minimum level in its $\nu$-string, we have 
$\gamma - \nu \notin \Delta$. Since $\gamma$ and $\nu$ are non-proportional, we have $\gamma - \nu \neq 0$. Altogether, we conclude  $\gamma - \nu 
\notin \Delta_0$.
	
	(ii): Using the Jacobi identity, $\gamma - \nu \notin \Delta_0$ and Lemma~\ref{lemma:a,k0}(\ref{X:X}), we obtain
	\[
	[\theta \xi, [\xi, X]]  =  - [X, [\theta \xi, \xi]] - [\xi, [X, \theta 
\xi]] = [[\theta \xi, \xi], X]
	= 2[H_{\nu}, X] = A_{\nu, \gamma} |\nu|^2 X.
	\]
	
	(i): Let $Y \in \g{g}_{\gamma}$. Combining (\ref{cartan:inner}) with (\ref{lemma:ad:ii}), we obtain
	\begin{align*}
	\langle \ad(\xi)X, \ad(\xi)Y \rangle_{AN}  =& - \langle X, \ad(\theta \xi) \circ \ad(\xi)Y \rangle_{AN} = -\langle X, [\theta \xi, [\xi, Y]] \rangle_{AN} \\
	 = & -|\nu|^2 A_{\nu, \gamma} \langle X, Y \rangle_{AN}.
	\end{align*}
	Since the $\nu$-string of $\gamma$ is non-trivial, we have $A_{\nu, \gamma} < 0$ and the assertion follows. 
	
	(iii): Next, using the Jacobi identity, (\ref{lemma:ad:ii}) and Lemma~\ref{lemma:a,k0}(\ref{X:X}), we deduce
	\begin{align*}
	[\theta \xi, [\xi, [\xi, X]]] & = -[[\xi, X], [\theta \xi, \xi]] - [\xi, [[\xi, X], \theta \xi]] \\
	& = 2 [H_{\nu}, [\xi, X]] + [\xi, [\theta \xi, [\xi, X]]] 
	=  (A_{\nu, \gamma + \nu} +A_{\nu, \gamma})|\nu|^2 [\xi, X].
	\end{align*}
	
	(iv): Similar arguments as those used before, together with (\ref{lemma:ad:iii}), give
	\begin{align*}
	[\theta \xi, [\xi, [\xi, [\xi, X]]]] 	&= - [[\xi, [\xi, X]], [\theta \xi, \xi]] - [\xi, [[\xi, [\xi, X]], \theta \xi]] \\
	& =  2[H_{\nu}, [\xi, [\xi, X]]] + [\xi, [\theta \xi, [\xi, [\xi, X]]]]  \\
	& = A_{\nu, \gamma + 2\nu}|\nu|^{2} [\xi, [\xi, X]] 
	+ (A_{\nu, \gamma + \nu} +A_{\nu, \gamma})|\nu|^2 [\xi, [\xi, X]] \\ 
	& = (A_{\nu, \gamma + 2\nu} + A_{\nu, \gamma+\nu} + A_{\nu, \gamma}) |\nu|^{2}[\xi,[\xi, X]]. 
	  \end{align*} 
This finishes the proof.
\end{proof}

\section{Construction of CPC submanifolds} \label{construction:examples}

In this section we construct new examples of CPC submanifolds in the rank 
$2$ symmetric spaces of non-compact type $SL_3(\mathbb{R})/SO_3$, $SL_3(\mathbb{C})/SU_3$, $SL_3(\mathbb{H})/Sp_3\ (=SU^*_6/Sp_3)$ and  $E^{-26}_6/F_4$. These spaces are precisely the non-compact symmetric spaces whose restricted root system is of type $A_2$. The new examples will provide the building blocks for further new examples in other non-compact symmetric spaces, via the so-called canonical extension method introduced in \cite{BT13} and studied further in \cite{miguel}. We emphasize that the CPC property is not preserved in general under the canonical extension method 
(an example will be given later). A fundamental ingredient in our investigations will be a decomposition of the tangent space of a CPC submanifold 
into subspaces that are invariant under the shape operator.

We continue using the notations introduced in Section~\ref{preliminaries}. Our construction is based on a suitable choice of a linear subspace $V$ of $\bigoplus_{\alpha \in \Pi'} \g{g}_{\alpha} \subseteq \g{n}$. The nilpotent subalgebra $\g{n}$ has a natural gradation that is generated by $\bigoplus_{\alpha \in \Pi} \g{g}_{\alpha}$. Thus, if we remove a linear subspace $V$ from $\bigoplus_{\alpha \in \Pi'} \g{g}_{\alpha}$, that is, consider the subspace $\g{n} \ominus V$, we get a subalgebra of $\g{n}$. We then define the subspace
\[
\g{s} = \g{a} \oplus (\g{n} \ominus V)
\]
of $\g{a} \oplus \g{n}$. Unfortunately, this subspace is in general not a 
subalgebra of $\g{a} \oplus \g{n}$. Assume for the moment that $\g{s}$ is 
a subalgebra of $\g{a} \oplus \g{n}$ and choose a vector $X \in \g{s}$ of 
the form $X = \sum_{\alpha \in \Pi'} X_{\alpha}$. Let $\beta \in \Pi'$ and
$0 \neq H \in \g{a} \ominus \left(\bigoplus_{\alpha \in \Pi \backslash \{\beta\}} \mathbb{R} H_{\alpha}\right)$.
Since $\g{s}$ is a subalgebra of $\g{a} \oplus \g{n}$ and $\g{a} \subset \g{s}$, we get $[H, X]  =  \sum_{\alpha \in \Pi'} [H,X_\alpha] =  \sum_{\alpha \in \Pi'} \alpha(H)X_\alpha = \beta (H) X_{\beta} \in \g{s}$. 
Since $H \neq 0$ is orthogonal to the vector spaces $\mathbb{R} H_{\alpha}$ for all $\alpha \in \Pi \backslash \{ \beta \}$, we must have $\beta (H) \neq 0$ and hence $X_\beta \in \g{s}$. 
Thus, if $\sum_{\alpha \in \Pi'} X_{\alpha} \in \g{s}$, we deduce that $X_{\alpha} \in \g{s}$ for all $\alpha \in \Pi'$. Consequently, if $\g{s}$ is a subalgebra of $\g{a} \oplus \g{n}$, then $V$ is of the form 
\begin{equation} \label{V:decomposition}
V = \bigoplus_{\alpha \in \psi} V_{\alpha}
\end{equation}
with $V_{\alpha} \subseteq \g{g}_{\alpha}$ and $\psi \subseteq \Pi'$. Without loss of generality, we can assume that $V_{\alpha} \neq \{0\}$ for each $\alpha \in \psi$. 

We assume from now on that $\g{s} = \g{a} \oplus (\g{n} \ominus V)$ is a subalgebra of $\g{a} \oplus \g{n}$ and that $V$ is of the form (\ref{V:decomposition}).  Let $S$ be the connected closed subgroup of $AN$ with Lie algebra $\g{s}$. The orbit $S \cdot o$ of $S$ through $o$ is a connected homogeneous submanifold of the symmetric space $M = G/K \cong AN$. We want to understand when this orbit is a CPC submanifold.

The simplest situation occurs when $V$ is contained in a single root space $\g{g}_\alpha$, $\alpha \in \Pi^\prime$. Let $m_\alpha = \dim(\g{g}_\alpha)$ and $k = m_\alpha - \dim(V)$. The orbit through $o$ of the connected closed subgroup of $AN$ with Lie algebra ${\mathbb R}H_\alpha \oplus \g{g}_\alpha$ is a real hyperbolic space ${\mathbb R}H^{m_\alpha + 1}$, 
embedded in $M$ as a totally geodesic submanifold. The orbit through $o$ of the connected closed subgroup of $AN$ with Lie algebra ${\mathbb R}H_\alpha \oplus (\g{g}_\alpha \ominus V)$ is a real hyperbolic space ${\mathbb R}H^{k + 1}$, embedded in ${\mathbb R}H^{m_\alpha + 1}$ as a totally geodesic submanifold. This  ${\mathbb R}H^{k + 1}$ is the singular orbit of a cohomogeneity one action on ${\mathbb R}H^{m_\alpha + 1}$. This cohomogeneity one action admits a canonical extension to a cohomogeneity one action on $M$ (see \cite{BT13} for details). The singular orbit of this cohomogeneity one action on $M$, which is the canonical extension of ${\mathbb R}H^{k + 1}$, then must be a CPC submanifold since the slice representation at any point of the singular orbit acts transitively on the unit sphere in the normal bundle.  
We can also give a slightly more complicated argument in this situation, which has the advantage though that we can apply it to more general situations. The generic orbits are homogeneous hypersurfaces, hence have the properties that they are both isoparametric and have constant principal curvatures. By applying the result by Ge and Tang that we mentioned in the introduction, we can deduce that the canonical extension of the ${\mathbb 
R}H^{k + 1}$ must be a CPC submanifold. It is this line of argument that we are going to apply for producing our new examples.

Back to the general situation. The orbit $S \cdot o$ is  a homogeneous submanifold and therefore it suffices to study its shape operator $\Ss$ at the point $o$. We will now investigate the shape operator in Lie algebraic terms by using equation (\ref{levi}). In our situation we need to analyze the equation
\begin{equation}\label{levi:examples}
4 \langle \nabla_{X} \xi, Z  \rangle_{AN} =  \langle [X, \xi] + (1-\theta)[\theta X, \xi], Z  \rangle_{B_{\theta}}
\end{equation}
for unit normal vectors $\xi \in V$, tangent vectors $X \in \g{s}$ and all $Z \in \g{a} \oplus \g{n}$.

We start by choosing $X \in \g{a} \subset \g{p}$. Then $\theta X = -X$ and
\[
[X, \xi] + [\theta X, \xi] - [X, \theta \xi] = - [X, \theta \xi] \in \bigoplus_{\alpha \in \psi} \g{g}_{-\alpha}.  
\]
Hence $[X, \theta \xi]$ has trivial projection onto $\g{a} \oplus \g{n}$. 
Therefore, $\nabla_{X} \xi = 0$ for all $X \in \g{a}$ and all normal vectors $\xi \in V$. In other words, for each unit normal vector $\xi$, $0$ 
is a principal curvature of $S \cdot o$ with respect to $\xi$ and $\g{a}$ 
is contained in the $0$-eigenspace. This is also clear from a geometric viewpoint. The orbit $A \cdot o$ is a Euclidean space ${\mathbb E}^r$ of dimension $r = {\rm rk}(M)$ and embedded in $M$ as a totally geodesic flat submanifold, a so-called maximal flat in $M$. Since $\g{a} \subset \g{s}$, we have $A \cdot o \subset S \cdot o$, and the assertion follows.
In particular, the maximal flat $A \cdot o = {\mathbb E}^r$ is a totally geodesic submanifold of $S \cdot o$. 

Therefore, we now need to examine the terms involved in (\ref{levi:examples}) when $X \in \g{n} \ominus V$. On the one hand, since $X,\xi \in\g{n}$, we have $[X, \xi] \in \g{n}$ and hence $[X, \xi]$ has trivial projection onto $\g{a}$. On the other hand, we will see that the elements $[\theta X, \xi]$ and $[X, \theta \xi]$ involved in (\ref{levi:examples}) have also trivial projections onto $\g{a}$. Moreover, we will justify that $[\theta X, \xi]$ must have trivial projection onto $\g{a} \oplus \g{n}$. 

Let $X \in \g{n} \ominus V$ and decompose $X$ into  $X = \sum_{\lambda \in \Delta^{+}} X_{\lambda}$ with $X_{\lambda} \in \g{g}_{\lambda}$. We decompose $\xi$ into $\xi = \sum_{\alpha \in \psi} \xi_{\alpha}$ with $\xi_{\alpha} \in V_{\alpha}$. Let $\alpha \in \psi$ and $\beta \in \Pi$. We will analyze the elements $[\theta  X_{\beta}, \xi_{\alpha} ]$ and $[X_{\beta}, \theta \xi_{\alpha}]$. Since $\alpha,\beta \in \Pi$, we have $\pm(\alpha - \beta) \notin \Delta$. Using (\ref{bracket:relation}) we deduce $[\theta  X_{\beta}, \xi_{\alpha} ] = 0 = [X_{\beta}, \theta \xi_{\alpha}]$ whenever $\alpha \neq \beta$. If $\beta = \alpha$, since $\langle X_{\alpha}, \xi_{\alpha} \rangle_{AN} = 0$ for all $\alpha \in \psi$ because of (\ref{V:decomposition}), we have $[\theta  X_{\alpha}, \xi_{\alpha} ] \in \g{k}_0$ and $[X_{\alpha}, \theta \xi_{\alpha}] \in \g{k}_0$ by Lemma \ref{lemma:a,k0}(\ref{ak0}), and hence they have trivial projections onto $\g{a} \oplus \g{n}$. Thus we conclude that $[\theta X_\beta, 
\xi]$ and $[X_\beta, \theta \xi]$ have trivial projections onto $\g{a} \oplus \g{n}$. Let $\lambda \in \Delta^+\setminus\Pi$. Then $\alpha - \lambda \notin \Delta^+_0$ and hence $[\theta X_{\lambda}, \xi_{\alpha}] \in \g{g}_{\alpha -\lambda}$ has trivial projection onto $\g{a} \oplus \g{n}$. 
Altogether this implies that $[\theta X, \xi]$ has trivial projection onto $\g{a} \oplus \g{n}$. 
We also see that $[X_\lambda, \theta \xi_{\alpha}] \in \g{g}_{\lambda - \alpha}$ has trivial projection onto $\g{g}_0$ and, since $\g{a} \subseteq 
\g{g}_0$, also onto $\g{a}$, which implies that $[X, \theta \xi]$ has trivial projection onto $\g{a}$. Then the Levi-Civita connection becomes
\begin{equation*}
2\langle \nabla_{X} \xi, H+ Y \rangle_{AN} =  \langle [X, \xi] - [X, \theta \xi], Y \rangle_{AN}
\end{equation*}
for $H \in \g{a}$, $Y \in \g{n}$, $\xi \in V$ and $X \in \g{n} \ominus V$. We saw above that $[X, \theta \xi] \in \g{k}_0 \oplus \g{n}$. Moreover, $0 \neq [X, \theta \xi] \in \g{k}_0$ is possible only if there exists $\alpha \in \psi$ with $X_\alpha \neq 0 \neq \xi_\alpha$. In this situation, since $X_\alpha,\xi_\alpha$ are orthogonal for each $\alpha \in \psi$, we have $\nabla_{X_\alpha}{\xi_\alpha} = 0$ by Lemma \ref{lemma:a,k0}(\ref{sh:itself}). Otherwise, the above equation yields
\begin{equation} \label{levi:connection:examples}
2 \nabla_{X} \xi =   [X, \xi] - [X, \theta \xi]
\end{equation}
for all $\xi \in V$ and $X \in \g{n} \ominus V$ with $[X_\alpha, \theta \xi_\alpha] = 0$ for all $\alpha \in \psi$.
In particular, if $\xi \in \g{g}_\alpha$ and $X \in \g{n} \ominus V$, equations (\ref{definition:shape:operator}) and (\ref{levi:connection:examples}) imply that the shape operator $\Ss_{\xi}$ with respect to $\xi$ of the submanifold $S \cdot o$ can be written as
\begin{equation}\label{shape}
2 \Ss_{\xi} X = -  \left([X, \xi] -[X, \theta \xi]\right)^{\top} =  [(1-\theta)\xi, X ]^{\top}.
\end{equation}
Note that $\theta(\xi - \theta \xi) = -(\xi - \theta \xi)$ and hence $\frac{1}{2}(1-\theta)\xi \in \g{p}$ is the orthogonal projection of $\xi$ onto $\g{p}$ with respect to $B_\theta$.

Before considering the examples introduced in the Main Theorem, we will study the behavior of the Levi-Civita connection in terms of the concept of string. Let $\gamma \in \Delta^{+}$ be the root of minimum level in its 
non-trivial $\nu$-string, for $\nu \in \Delta^{+}$ non-proportional to $\gamma$. For each unit vector $\xi \in \g{g}_{\nu}$ we define 
\begin{equation}\label{definition:phi}
\phi_{\xi} = |\nu|^{-1} (-A_{\nu, \gamma})^{-1/2} \ad(\xi)\ ,\ 
\phi_{\theta \xi} = -|\nu|^{-1} (-A_{\nu, \gamma})^{-1/2} \ad(\theta\xi).
%\begin{aligned}
%\phi_{\xi} &{}= |\nu|^{-1} (-A_{\nu, \gamma})^{-1/2} \ad(\xi),\\ 
%\phi_{\theta \xi} &{}= -|\nu|^{-1} (-A_{\nu, \gamma})^{-1/2} \ad(\theta\xi).
%\end{aligned}
\end{equation}
From Lemma~\ref{lemma:ad}(\ref{lemma:ad:i}),(\ref{lemma:ad:ii}) we easily deduce: 

\begin{lemma}\label{lemma:phi:injective}
	Let $\gamma \in \Delta^{+}$ be the root of minimum level in its non-trivial $\nu$-string, for $\nu \in \Delta^{+}$ non-proportional to $\gamma$. Then:
	\begin{enumerate}[{\rm (i)}]
		\item $\phi_{\xi} \rvert_{\g{g}_{\gamma}} \colon \g{g}_{\gamma} \to  \g{g}_{\gamma + \nu}$ is a linear isometry onto $\phi_\xi(\g{g}_\gamma) = 
[\xi,\g{g}_\gamma]$.
		\item $(\phi_{\theta \xi} \circ \phi_{\xi})\rvert_{\g{g}_{\gamma}} = \id_{\g{g}_{\gamma}}$.
	\end{enumerate}
\end{lemma}

The next result will be useful for calculating principal curvatures explicitly.

\begin{proposition}\label{proposition:main1}
	Let $\gamma \in \Delta^{+}$ be the root of minimum level in its $\nu$-string, for $\nu \in \Delta^{+}$ satisfying $A_{\nu, \gamma} = -1$, and $\xi \in \g{g}_{\nu}$ be a unit vector with respect to $\langle \cdot, \cdot \rangle_{AN}$. Then $\phi_{\xi}$ and $\phi_{\theta \xi}$ are inverse linear isometries when restricted to $\g{g}_{\gamma}$ and $\g{g}_{\gamma+\nu}$, respectively. Moreover, for each $X \in \g{g}_{\gamma}$ we have
	\[
	\nabla_X \xi = - \frac{|\nu|}{2} {\phi_{\xi}(X)}\  \mbox{  and  } \ 
	\nabla_{\phi_{\xi}(X)} \xi = - \frac{|\nu|}{2} X.
	\]
\end{proposition}

\begin{proof}
	From Lemma~\ref{root:spaces:dimension} we deduce $\dim(\g{g}_{\gamma}) = 
\dim(\g{g}_{\gamma + \nu})$. Lemma~\ref{lemma:phi:injective} then implies 
that $\phi_{\xi} \rvert_{\g{g}_{\gamma}} \colon \g{g}_{\gamma} \to  \g{g}_{\gamma + \nu}$ is a linear isometry onto $\g{g}_{\gamma + \nu}$ and $(\phi_{\xi} \rvert_{\g{g}_{\gamma}})^{-1} = \phi_{\theta\xi} \rvert_{\g{g}_{\gamma+\nu}}$. Since $\gamma$ is the root of minimum level in its $\nu$-string and $A_{\nu, \gamma} = -1$, we have $\gamma - \nu \notin \Delta_0$. Using  (\ref{levi:connection:examples}), the fact that $\gamma - \nu \notin \Delta_0$, and then (\ref{definition:phi}), we deduce 
	\[
	2 \nabla_X \xi =   [X, \xi] - [X, \theta \xi] =  [X, \xi] = - |\nu| {\phi_{\xi}(X)},
	\]
	for  unit vectors $\xi \in \g{g}_{\nu}$ and vectors $X \in \g{g}_{\gamma}$. Finally, using  (\ref{levi:connection:examples}), the fact that $\gamma + 2 \nu \notin \Delta$,  (\ref{definition:phi}) and  then Lemma~\ref{lemma:phi:injective}, we obtain
	\begin{align*}
	2 \nabla_{\phi_{\xi}(X)} \xi = &   [\phi_{\xi}(X), \xi] - [\phi_{\xi}(X), \theta \xi] = - [\phi_{\xi}(X), \theta \xi]\\ = &  - |\nu| (\phi_{\theta \xi} \circ \phi_{\xi}) (X) = -|\nu| X,
	\end{align*}
	for a unit vector $\xi \in \g{g}_{\nu}$ and $X \in \g{g}_{\gamma}$.
\end{proof}

After these considerations we shall focus now on the examples introduced in the Main Theorem. Consider a symmetric space $G/K$ of non-compact type 
with at least two simple roots, say $\alpha_0$ and $\alpha_1$, that are connected by a single edge in its Dynkin diagram.  
Consider the subalgebra $\g{s}  = \g{a} \oplus (\g{n} \ominus V)$ with $\psi = \{\alpha_0, \alpha_1\}$ and $V \subseteq \g{g}_{\alpha_0} \oplus \g{g}_{\alpha_1}$. Let $\xi \in V$ be a unit vector and $X \in \g{g}_{\lambda}^{\top}$, where $\g{g}_{\lambda}^{\top}$ denotes the orthogonal projection of $\g{g}_{\lambda}$ onto $\g{n} \ominus V$ for $\lambda \in \Delta^{+}$. From (\ref{shape}) and (\ref{bracket:relation}) we obtain
\begin{equation*}
\Ss_{\xi} X  \in (\g{g}_{\lambda + \alpha_0}^{\top} \oplus \g{g}_{\lambda 
+ \alpha_1}^{\top}) \oplus (\g{g}_{\lambda - \alpha_0}^{\top} \oplus \g{g}_{\lambda - \alpha_1}^{\top}).
\end{equation*}
This shows that we need to understand how the shape operator $\Ss$ relates the different root spaces of positive roots. 

In order to clarify this situation, we introduce a generalization of the concept of $\alpha$-string. For ${\alpha_0}, {\alpha_1} \in \Delta$ and $\lambda \in \Delta_0$ we define the $({\alpha_0}, {\alpha_1})$-string containing $\lambda$ as the set of all elements in $\Delta_0$ of the form $\lambda + n {\alpha_0} + m {\alpha_1}$ with $n,m \in \mathbb{Z}$. This leads to the following equivalence relation on $\Delta^{+}$. We say that two 
roots $\lambda_1, \lambda_2 \in \Delta^{+}$ are $(\alpha_0, \alpha_1)$-related if $\lambda_1 - \lambda_2 = n {\alpha_0} + m {\alpha_1}$ for some 
$n,m \in \mathbb{Z}$. Therefore, the equivalence class $[\lambda]_{(\alpha_0, \alpha_1)}$ of the root $\lambda \in \Delta^{+}$ consists of the elements which may be written as $\lambda +n {\alpha_0} + m {\alpha_1}$ for some $n,m \in \mathbb{Z}$. We will write  $[\lambda]$ for this equivalence class, taking into account that this class depends on the roots $\alpha_0$ and $\alpha_1$ defining the string. Put $\Delta^{+} / \sim$ for the set of equivalence classes. The family $\{[\lambda]\}_{\lambda \in \Delta^{+}}$ constitutes a partition of $\Delta^{+}$. 

Using this notation, we can now write
\begin{equation}\label{eq:shapeclasses}
\Ss_{\xi} \left( \bigoplus_{\gamma \in [\lambda]} \g{g}_{\gamma}^{\top} \right) \subseteq \bigoplus_{\gamma \in [\lambda]} \g{g}_{\gamma}^{\top} \mbox{ for all } \lambda \in \Delta^{+}.
\end{equation}
In other words, for each $\lambda \in \Delta^{+}$ the subspace $\bigoplus_{\gamma \in [\lambda]} \g{g}_{\gamma}^{\top}$ is an $\Ss_{\xi}$-invariant subspace of the tangent space $\g{s}$. Clearly, $S \cdot o$ is a CPC submanifold if and only if the eigenvalues of $\Ss_\xi$ are independent of the unit normal vector $\xi$ when restricted to each of these invariant subspaces  $\bigoplus_{\gamma \in [\lambda]} \g{g}_{\gamma}^{\top}$ for every $\lambda \in \Delta^{+}$. Thus it suffices to consider the orthogonal 
decomposition
\begin{equation}\label{invariant:decomposition}
\g{n} \ominus V  = \bigoplus_{\lambda \in \Delta^{+} / \sim} \left(\bigoplus_{\gamma \in [\lambda]} \g{g}_{\gamma}^{\top} \right)
\end{equation}
and to study the shape operator when restricted to each of these $\Ss_{\xi}$-invariant subspaces. These invariant subspaces will be determined more explicitly in Lemma \ref{structure:strings} by using the concept of $(\alpha_0, \alpha_1)$-string of $\lambda$. However, note that one of them is very easy to determine. Since $\alpha_0$ and $\alpha_1$ are simple roots and connected by a single edge in the Dynkin diagram, the $(\alpha_0, \alpha_1)$-string of $\alpha_0$ is just the set of roots of a rank 2 symmetric space of non-compact type whose Dynkin diagram is of type $A_2$. Therefore, studying the shape operator $\Ss_{\xi}$ when restricted to the $\Ss_{\xi}$-invariant subspace $\bigoplus_{\gamma \in [\alpha_0]} \g{g}_{\gamma}^{\top}$ is equivalent to studying the CPC property of the submanifold $S \cdot o$ in one of the symmetric spaces $SL_{3}(\mathbb{R})/SO_{3}$, $SL_{3}(\mathbb{C})/SU_{3}$, $SL_{3}(\mathbb{H})/Sp_{3}$ or  $E^{-26}_6/F_4$. The remaining part of this section is devoted to the study of the shape operator of $S \cdot o$ when restricted to the vector space $\bigoplus_{\gamma \in [\alpha_0]} \g{g}_{\gamma}^{\top}$, or equivalently, to classifying CPC submanifolds in these rank $2$ symmetric spaces under the hypotheses of the Main Theorem.

We restrict now to the rank $2$ symmetric spaces of non-compact type whose Dynkin diagram is of type $A_2$. In this case  we have $\Delta^{+} = \{\alpha_0, \alpha_1, \alpha_0 + \alpha_1\}$ and $|\alpha_0| = |\alpha_1|=|\alpha_0 + \alpha_1| = \sqrt{2}$. Moreover, from Lemma \ref{root:spaces:dimension} we see that $\dim(\g{g}_{\alpha_0}) =\dim(\g{g}_{\alpha_1}) = \dim(\g{g}_{\alpha_0+\alpha_1})$. In line with the construction that we explained at the beginning of this section, we consider the subalgebra 
\[
\g{s} = \g{a} \oplus (\g{g}_{\alpha_0} \ominus V_{\alpha_0}) \oplus (\g{g}_{\alpha_1} \ominus V_{\alpha_1}) \oplus \g{g}_{\alpha_0 + \alpha_1}
\]
with $V = V_{\alpha_0} \oplus V_{\alpha_1}$ and $\{0\} \neq V_{\alpha_k} \subseteq \g{g}_{\alpha_k}$, $k \in \{0,1\}$. We put $V_k = V_{\alpha_k}$ and $T_k = \g{g}_{\alpha_k} \ominus V_k$. If $U_1,U_2$ are linear subspaces of $\g{g}$, we denote by $[U_1, U_2]$ the linear subspace of $\g{g}$ spanned by  $\{[u_1, u_2] : u_1 \in U_1, u_2 \in U_2 \}$. The following result will help us computing the shape operator of $S \cdot o$ explicitly.

\begin{lemma}\label{lemma:auxiliar:characterisation}
	Let $0 \neq \xi_k \in V_k$, $k \in \{0,1\}$. Then  
	\begin{equation}\label{lemma:auxiliar:characterisation:decomposition}
	\g{g}_{\alpha_0 + \alpha_1} = \phi_{\xi_0} (V_1) \oplus \phi_{\xi_0} (T_1)  = \phi_{\xi_1} (V_0) \oplus \phi_{\xi_1} (T_0)
	\end{equation}
	are orthogonal decompositions of $\g{g}_{\alpha_0 + \alpha_1}$. Moreover, if $\dim(V_0) = \dim(V_1) = \dim([V_0, V_1])$, then:
	\begin{enumerate}[{\rm (i)}]
		\item  $\phi_{\xi_0} (V_1) = \phi_{\xi_1} (V_0) = [V_0, V_1]$ and \\ $\phi_{\xi_0} (T_1) = \phi_{\xi_1} (T_0) = [V_0, T_1] = [V_1, T_0]$. \label{lemma:auxiliar:characterisation:i}
		\item If $T_k \neq \{0\}$, then $\dim(T_k) \geq \dim(V_k)$. \label{lemma:auxiliar:characterisation:iii}
		\item The maps $(\phi_{\theta \xi_0} \circ \phi_{\xi_1})\rvert_{T_0} \colon T_0 \to T_1$ and $(\phi_{\theta \xi_1} \circ \phi_{\xi_0})\rvert_{T_1} \colon T_1 \to T_0$ are linear isometries and
		\[
		\g{s} \ominus (\g{a} \oplus V)  = T_0 \oplus T_1 \oplus [V_0,T_1] \oplus [V_0, V_1] 
		\]
		is an orthogonal decomposition of $\g{s} \ominus (\g{a} \oplus V)$. \label{lemma:auxiliar:characterisation:ii}
	\end{enumerate}
\end{lemma}

\begin{proof}
	According to Proposition~\ref{proposition:main1}, the map $\phi_{\xi_0} \rvert_{\g{g}_{\alpha_1}} \colon \g{g}_{\alpha_1} \to \g{g}_{\alpha_0 + \alpha_1}$ and the map $\phi_{\xi_1} \rvert_{\g{g}_{\alpha_0}} \colon \g{g}_{\alpha_0} \to \g{g}_{\alpha_0 + \alpha_1}$ are linear isometries. Since $\g{g}_{\alpha_k} = V_k \oplus T_k$ is an orthogonal decomposition by 
construction, we get (\ref{lemma:auxiliar:characterisation:decomposition}).
	
	Assume from now on that $\dim(V_0) = \dim(V_1) = \dim([V_0, V_1])$. Since we have $\phi_{\xi_0} (V_1) \subseteq [V_0,V_1]$ and we have $\dim(\phi_{\xi_0} (V_1)) = \dim(V_1) = \dim([V_0, V_1])$, we get $\phi_{\xi_0} (V_1) = [V_0,V_1]$, and analogously, $\phi_{\xi_1} (V_0) = [V_0,V_1]$. From (\ref{lemma:auxiliar:characterisation:decomposition}) we then 
obtain the other part of (\ref{lemma:auxiliar:characterisation:i}). From (i) we get $\dim(T_0) = \dim([T_1,V_0])$. If $\dim(T_1) > 0$ we also get $\dim([T_1,V_0]) \geq \dim(V_0)$ from Proposition~\ref{proposition:main1}. Altogether this implies $\dim(T_0) \geq \dim(V_0)$ if $\dim(T_1) > 0$. Analogously, we get $\dim(T_1) \geq \dim(V_1)$ if $\dim(T_0) > 0$. Note that $\dim(T_0) = \dim(T_1)$. This proves (\ref{lemma:auxiliar:characterisation:iii}).
	Recall that $\phi_{\xi_0}(T_1)=\phi_{\xi_1}(T_0)$ is orthogonal to  $\phi_{\xi_0}(V_1) = \phi_{\xi_1}(V_0)$. For $X_0 \in T_0$ and $\eta_1 \in V_1$ we have
	\[
	\langle (\phi_{\theta \xi_0} \circ \phi_{\xi_1})(X_0), \eta_1 \rangle_{AN}  =  \langle \phi_{\xi_1}(X_0), \phi_{\xi_0}(\eta_1) \rangle_{AN} = 
0.
	\]
	Since $\dim(V_0) = \dim(V_1)$ and $\dim(T_0) = \dim(T_1)$, then $(\phi_{\theta \xi_0} \circ \phi_{\xi_1})\rvert_{T_0} \colon T_0 \to T_1$ is a linear isometry. This implies (\ref{lemma:auxiliar:characterisation:ii}). 
\end{proof}

The next result provides an algebraic characterization of the CPC property of the orbit $S \cdot o$.

\begin{proposition}\label{characterisation}
	Let $\g{s}$ be the subalgebra of $\g{a} \oplus \g{n}$ defined by 
	\[
	\g{s} = \g{a} \oplus (\g{g}_{\alpha_0} \ominus V_0) \oplus (\g{g}_{\alpha_1} \ominus V_1) \oplus \g{g}_{\alpha_0+\alpha_1}
	\]
	and $S$ be the connected closed subgroup of $AN$ with Lie algebra $\g{s}$. Then the orbit $S \cdot o$ is a CPC submanifold of the symmetric space 
$G/K = AN$ if and only if $\dim(V_0) = \dim(V_1) = \dim([V_0, V_1])$. 
	Moreover, if $S \cdot o$ is a CPC submanifold, then its principal curvatures are $\pm \frac{1}{\sqrt{2}}$, both with multiplicity $\dim(T_0)$, and $0$ with multiplicity $\dim(\g{g}_{\alpha_0 + \alpha_1}) + 2$. 
\end{proposition}

\begin{proof}
	Assume that the orbit $S \cdot o$ is a CPC submanifold. Let $j,k \in \{0,1\}$ with $j \neq k$ and $\xi_j \in V_j$ be a unit vector. According to (\ref{lemma:auxiliar:characterisation:decomposition}), the tangent space $\g{s}$ of $S \cdot o$ at $o$ has the orthogonal decomposition
	\[
	\g{s} = \g{a} \oplus T_j \oplus T_k \oplus \phi_{\xi_j}(T_k) \oplus \phi_{\xi_j}(V_k).
	\]
	We saw at the beginning of this section that $\Ss_{\xi_j}\rvert_{\g{a}} = 0$. Using Lemma \ref{lemma:a,k0}(\ref{sh:itself}) and Proposition~\ref{proposition:main1}, we get following expression for the shape operator $\Ss_{\xi_j}$:
	\[
	\sqrt{2} \Ss_{\xi_j} X = \phi_{\xi_j}(X_{T_k}) + \phi_{\theta \xi_j}(X_{\phi_{\xi_j}(T_k)}) ,
	\]
	where $X \in \g{s}$ is a tangent vector and the index to $X$ denotes the 
orthogonal projection of $X$ onto that space. In particular, $\dim(\ker(\Ss_{\xi_j})) = 2 + \dim(T_j) + \dim(V_k)$. As $S \cdot o$ is a CPC submanifold, we have $\dim(\ker(\Ss_{\xi_j}))  = \dim(\ker(\Ss_{\xi_k}))$ and thus we get $\dim(T_j) + \dim(V_k) = \dim(T_k) + \dim(V_j)$. On the other hand, we have 
	\[
	\dim(T_j) + \dim(V_j) = \dim(\g{g}_{\alpha_j}) = \dim(\g{g}_{\alpha_k}) = \dim(T_k) + \dim(V_k).
	\] 
	From the previous two equations we easily get $\dim(V_j) = \dim(V_k)$, 
that is, $\dim(V_0) = \dim(V_1)$ (and then also $\dim(T_0) = \dim(T_1)$).
	
	We now investigate the shape operator $\Ss_\xi$ with respect to the unit 
normal vector $\xi = \frac{1}{\sqrt{2}}(\xi_0 + \xi_1)$. Since $\Ss_\xi 
= \frac{1}{\sqrt{2}}(\Ss_{\xi_0} + \Ss_{\xi_1})$, we get
	\[
	2\Ss_{\xi} X = \phi_{\xi_0}(X_{T_1}) + \phi_{\xi_1}(X_{T_0}) +  \phi_{\theta \xi_0}(X_{\phi_{\xi_0}(T_1)}) + \phi_{\theta \xi_1}(X_{\phi_{\xi_1}(T_0)}) .
	\]
	Since all the $\phi$-maps are linear isometries on the corresponding spaces (see Lemma \ref{lemma:phi:injective}), we obtain  
	\begin{align*}
	\ker(\Ss_\xi) ={}& \g{a} \oplus \{X \in T_0 \oplus T_1 : \phi_{\xi_0}(X_{T_1}) + \phi_{\xi_1}(X_{T_0}) = 0\}\\
	& \oplus \{X \in \g{g}_{\alpha_0 + \alpha_1} : X_{\phi_{\xi_0}(T_1)} = 
0 = X_{\phi_{\xi_1}(T_0)}\} \\
	= {} & \g{a} \oplus \{ \phi_{\theta\xi_0}Y - \phi_{\theta\xi_1}Y \in T_0 \oplus T_1 : Y \in \phi_{\xi_0}(T_1) \cap \phi_{\xi_1}(T_0) \} \\
	& \oplus (\phi_{\xi_0}(V_1) \cap \phi_{\xi_1}(V_0)).
	\end{align*}
	Since $S \cdot o$ is a CPC submanifold, then $\dim(\ker(\Ss_{\xi_j})) = 
\dim(\ker(\Ss_{\xi}))$ and therefore
	\[
	\dim(T_j) + \dim(V_k) = \dim(\phi_{\xi_0}(T_1) \cap \phi_{\xi_1}(T_0)) 
+ \dim(\phi_{\xi_0}(V_1) \cap \phi_{\xi_1}(V_0)).
	\]
	Again, since all the $\phi$-maps are linear isometries on the corresponding spaces, this is possible only when $\phi_{\xi_0}(T_1) = \phi_{\xi_1}(T_0)$ and $\phi_{\xi_0}(V_1) = \phi_{\xi_1}(V_0)$. As $\xi_0 \in V_0$ 
and $\xi_1 \in V_1$ are arbitrary unit vectors, this implies in particular that $\dim([V_0,V_1]) = \dim(V_0) = \dim(V_1)$. 	
	
	Conversely, assume that $\dim([V_0, V_1]) = \dim(V_0) = \dim(V_1)$. Let $\xi$ be a unit normal vector of $S \cdot o$ at $o$. There exist unit 
vectors $\xi_0 \in V_0$, $\xi_1 \in V_1$ and $\varphi \in [0,\frac{\pi}{2}]$ so that $\xi = \cos(\varphi) \xi_0 + \sin(\varphi) \xi_1$. From Lemma~\ref{lemma:auxiliar:characterisation} we have the orthogonal decomposition
	\begin{equation}\label{characterization:decomposition}
	\begin{aligned}
	\g{s} = {}& \g{a} \oplus T_0 \oplus (\phi_{\theta \xi_0} \circ \phi_{\xi_1})(T_0) \oplus \phi_{\xi_1}(T_0) \oplus [V_0, V_1] \\
	=  {}& \g{a} \oplus T_0 \oplus T_1 \oplus [V_1,T_0] \oplus [V_0, V_1] 
	\end{aligned}
	\end{equation}
	of the tangent space $\g{s}$ of $S \cdot o$ at $o$. As shown above, we have
	\[
	\sqrt{2} \Ss_{\xi_j} X = \phi_{\xi_j}(X_{T_k}) + \phi_{\theta \xi_j}(X_{\phi_{\xi_j}(T_k)}) .
	\]
	This implies
	\begin{align*}
	\sqrt{2} \Ss_{\xi} X = & \cos(\varphi)(\phi_{\xi_0}(X_{T_1}) + \phi_{\theta \xi_0}(X_{\phi_{\xi_0}(T_1)}) ) \\& + \sin(\varphi)(\phi_{\xi_1}(X_{T_0}) + \phi_{\theta \xi_1}(X_{\phi_{\xi_1}(T_0)}) ).
	\end{align*}
	We immediately see that $\Ss_\xi$ vanishes on $\g{a} \oplus [V_0,V_1]$. Next, consider the vectors $0 \neq X \in T_0$, $\phi_{\xi_1}(X) \in [V_1,T_0] = [V_0,T_1]$ and $\phi_{\theta \xi_0}(\phi_{\xi_1}(X)) \in T_1$. The $3$-dimensional subspace of $\g{s}$ spanned by $X,\phi_{\xi_1}(X),\phi_{\theta \xi_0}(\phi_{\xi_1}(X))$ is $\Ss_\xi$-invariant and the matrix representation of $\Ss_\xi$ with respect to the basis $X$, $\phi_{\xi_1}(X)$, $\phi_{\theta \xi_0}(\phi_{\xi_1}(X))$ is 
	\[
	\frac{1}{\sqrt{2}} 
	\left(\begin{array}{ccc} 0& \sin(\varphi) & 0 \\   \sin(\varphi) & 0 & \cos(\varphi)  \\  0& \cos(\varphi) & 0  
	\end{array}\right).
	\]
	The eigenvalues of this matrix are $0$ and $\pm \frac{1}{\sqrt{2}}$. It follows that $S \cdot o$ is a CPC submanifold of $AN$. The statement about the principal curvatures and their multiplicities also follows from this calculation. 
\end{proof}

The previous result implies that the codimension of a CPC submanifold is even. However, as we will see in the next result, there are further constraints on the codimension.

\begin{corollary}\label{codimension}
	Let $\g{s}$ be the subalgebra of $\g{a} \oplus \g{n}$ defined by 
	\[
	\g{s} = \g{a} \oplus (\g{g}_{\alpha_0} \ominus V_0) \oplus (\g{g}_{\alpha_1} \ominus V_1) \oplus \g{g}_{\alpha_0+\alpha_1}
	\]
	and $S$ be the connected closed subgroup of $AN$ with Lie algebra $\g{s}$. Assume that $S \cdot o$ is a CPC submanifold of $G/K = AN$. 
	\begin{itemize}
		\item[(i)] If $G/K = SL_{3}(\mathbb{R})/SO_{3}$, then $S \cdot o$ has 
codimension $2$.
		\item[(ii)] If $G/K = SL_{3}(\mathbb{C})/SU_{3}$, then $S \cdot o$ has codimension $2$ or $4$.
		\item[(iii)] If $G/K = SL_{3}(\mathbb{H})/Sp_{3}$, then $S \cdot o$ has codimension $2$, $4$ or $8$.
		\item[(iv)] If $G/K = E^{-26}_6/F_4$, then $S \cdot o$ has codimension $2$, $4$, $8$ or $16$.
	\end{itemize}
\end{corollary}

\begin{proof}
	According to Proposition~\ref{characterisation}, for (i) and (ii) there is nothing to prove since the dimensions of the root spaces are $1$ and $2$ respectively. In the cases (iii) and (iv) the dimensions of the root spaces are $4$ and $8$ respectively, and therefore we need to exclude the possibility for codimension $6$ in case (iii) and for codimensions $6$, $10$, $12$ and $14$ in case (iv). The codimensions $10$, $12$ and $14$ in case (iv) cannot occur by Proposition~\ref{characterisation} and Lemma~\ref{lemma:auxiliar:characterisation}(\ref{lemma:auxiliar:characterisation:iii}).
	It remains to investigate the possibility for codimension $6$ in cases (iii) and (iv). In this situation we have $\dim(V_0) = \dim(V_1) = \dim([V_0,V_1]) = 3$.
	
	Let $\eta_1, \eta_2, \eta_3$ be an orthonormal basis of $V_1$ and $\xi_1$ be a unit vector in $V_0$. The vector $\xi_2 = (\phi_{\theta \eta_2} \circ \phi_{\eta_3})(\xi_1)$ is non-zero by means of Proposition~\ref{proposition:main1}. On the one hand, using again Proposition~\ref{proposition:main1}, we obtain 
	\begin{align*}
	\langle \xi_1, \xi_2 \rangle_{AN} & = \langle \xi_1, (\phi_{\theta \eta_2} \circ \phi_{\eta_3})(\xi_1) \rangle_{AN}= \langle \phi_{\eta_2}(\xi_1), \phi_{\eta_3}(\xi_1) \rangle_{AN}\\
	&=\langle \phi_{\xi_1}(\eta_2), \phi_{\xi_1}(\eta_3) \rangle_{AN}=\langle \eta_2, \eta_3 \rangle_{AN} = 0.
	\end{align*}
	On the other hand, we have $\phi_{\eta_2}(\xi_2) = (\phi_{\eta_2} \circ \phi_{\theta \eta_2} \circ \phi_{\eta_3})(\xi_1) = \phi_{\eta_3}(\xi_1)$. From Lemma~\ref{lemma:a,k0}(\ref{ak0}) we have $[\eta_3, \theta \eta_2] \in \g{k}_0 \subseteq \g{k}$. Since $\theta \rvert_{\g{k}} =~\id_{\g{k}}$ we have $[\eta_3, \theta \eta_2] = [\theta \eta_3, \eta_2]$. Using this and the Jacobi identity we get
	\[
	\phi_{\eta_3}(\xi_2) =  (\phi_{\eta_3} \circ \phi_{\theta \eta_2} \circ \phi_{\eta_3})(\xi_1) = - (\phi_{\eta_2} \circ \phi_{\theta \eta_3} \circ \phi_{\eta_3})(\xi_1) = -\phi_{\eta_2}(\xi_1).
	\]
	To sum up, having in mind definition (\ref{definition:phi}), we have shown that $\phi_{\xi_2}(\eta_2) =\phi_{\xi_1}(\eta_3)$ and $\phi_{\xi_2}(\eta_3) =  -\phi_{\xi_1}(\eta_2)$. Since $\phi_{\xi_1}(V_1)$ and $\phi_{\xi_2}(V_1)$ must be the same vector space by Proposition~\ref{characterisation} and Lemma~\ref{lemma:auxiliar:characterisation}(\ref{lemma:auxiliar:characterisation:i}), we conclude that $\phi_{\xi_1}(\eta_1)$ is either $\phi_{\xi_2}(\eta_1)$ or $-\phi_{\xi_2}(\eta_1)$, which implies that $\phi_{\eta_1}(\xi_1)$ is either $\phi_{\eta_1}(\xi_2)$ or $-\phi_{\eta_1}(\xi_2)$. Since $\langle \xi_1, \xi_2 \rangle_{AN} = 0$, this contradicts the injectivity of $\phi_{\eta_1}$ (see Proposition~\ref{proposition:main1}). This concludes the proof. 
\end{proof}

We want to derive a more geometric characterization of the CPC property. For this, we first prove an auxiliary result.

\begin{lemma}\label{k0}
	Let $X, Y \in \g{g}_{\alpha_0+\alpha_1}$ be orthonormal (and consequently $G/K \neq SL_{3}(\mathbb{R})/SO_{3}$). Then:
	\begin{enumerate} [{\rm (i)}]
		\item The linear map $\frac{1}{4} \ad([\theta X, Y])$ defines a complex 
structure on the vector space ${\mathbb R}X \oplus {\mathbb R}Y$ spanned by $X$ and $Y$. \label{k0:i}
		\item The linear map $\frac{1}{2}\ad([\theta X, Y])$ defines complex structures on the vector spaces $\g{g}_{\alpha_0}$ and $\g{g}_{\alpha_1}$. \label{k0:ii}
		\item Let $X, Y,Z \in \g{g}_{\alpha_0+\alpha_1}$ be orthonormal (and thus $G/K$ is $SL_{3}(\mathbb{H})/Sp_{3}$ or $E^{-26}_6/F_4$).  Define $J_1 
= \frac{1}{2} \ad([\theta X, Y])$, $J_2 =\frac{1}{2}\ad([\theta X, Z])$ and $J_3 = J_1 \circ J_2$. Then  $\{J_1, J_2, J_3\}$ defines quaternionic structures on the vector spaces $\g{g}_{\alpha_0}$ and $\g{g}_{\alpha_1}$. \label{k0:iii} \label{quaternionic:structure}
	\end{enumerate}
\end{lemma}

\begin{proof}
	(\ref{k0:i}): First, from the Jacobi identity, $2 (\alpha_0 + \alpha_1) \notin \Delta$ and Lemma~\ref{lemma:a,k0}(\ref{X:X}), we obtain
	\begin{equation}\label{complex:structure}
	[[\theta X, Y], X] = -[[X, \theta X], Y] = [[\theta X, X], Y] = 2 |\alpha_0 + \alpha_1|^2 Y = 4 Y.
	\end{equation}
	According to Lemma~\ref{lemma:a,k0}(\ref{ak0}) we have $[\theta X, Y] \in \g{k}_0 \subseteq \g{k}$. Since $\theta \rvert_{\g{k}} =~\id_{\g{k}}$ 
we have $[\theta X, Y] = [X, \theta Y]$. Together with (\ref{complex:structure}), we deduce that  $[[\theta X, Y], Y] = [[X, \theta Y], Y] = 
- [[\theta Y, X], Y] = -4X$. Thus we have $(\frac{1}{4}\ad([\theta X, Y]))^2 = -\mbox{id}$ on ${\mathbb R}X \oplus {\mathbb R}Y$.  
	
	(\ref{k0:ii}): Let $W \in \g{g}_{\alpha_k}$ for $k \in \{0,1\}$. Using the Jacobi identity, the equations (\ref{bracket:relation}),  (\ref{complex:structure}) and $[\theta X, Y] = [X, \theta Y]$, and Lemma \ref{lemma:a,k0}(\ref{X:X}), we obtain
	\begin{align*}
	[[\theta X, Y],[[\theta X, Y], W]] 
	&= - [[[[\theta X, Y], W], \theta X], Y] \\
	& = [[[W, \theta X],[\theta X, Y]], Y] + [[[\theta X,[\theta X, Y]], W], Y] \\
	& = [[[W, \theta X],[X, \theta Y]], Y] - [[\theta[[\theta X, Y], X], W], Y] \\
	& = -[[\theta Y, [[W, \theta X], X]], Y] - 4 [[\theta Y, W], Y]\\
	& = [[\theta Y, [[\theta X, X], W]], Y] - 4 [[\theta Y, Y], W] \\
	& = 2 [[\theta Y, [H_{\alpha_0+\alpha_1}, W]], Y] - 8 [H_{\alpha_0+\alpha_1}, W] 
	\\ &= 2 [[\theta Y, W], Y] - 8 W  = 2 [[\theta Y, Y], W] - 8 W 
	\\ &	= 4 [H_{\alpha_0+\alpha_1}, W] - 8 W  = 4 W - 8 W = -4 W. 
	\end{align*}

	(\ref{k0:iii}): With analogous arguments as above, we obtain
	\begin{align*}
	[[\theta X, Y],[[\theta X, Z], W]] & =  - [[[[\theta X, Z], W], \theta 
X], Y] \\
	& = [[[W, \theta X],[\theta X, Z]], Y] + [[[\theta X,[\theta X, Z]], W], Y] \\
	& = [[[W, \theta X],[X, \theta Z]], Y] - [[\theta[[\theta X, Z], X], W], Y] \\
	& = -[[\theta Z, [[W, \theta X], X]], Y] - 4 [[\theta Z, W], Y] \\
	& = [[\theta Z, [[\theta X, X], W]], Y] - 4 [[\theta Z, Y], W] \\
	& = 2 [[\theta Z, [H_{\alpha_0+\alpha_1}, W]], Y] + 4 [[\theta Y, Z], W] \\
	& = 2 [[\theta Z, W], Y] + 4 [[\theta Y, Z], W] \\
	& = -2 [[\theta Y, Z], W] + 4 [[\theta Y, Z], W] 
	=  2 [[\theta Y, Z], W].
	\end{align*}
	Using the previous equality and $[\theta Y, Z] = [Y, \theta Z]$, we deduce
	\begin{align*}
	 [[\theta X, Y],[[\theta X, Z], W]] & =  2 [[\theta Y, Z], W]  = -2 [[\theta Z, Y], W] \\ &
	= -[[\theta X, Z],[[\theta X, Y],W]].
	\end{align*}
	Now define $J_1 = \frac{1}{2} \ad([\theta X, Y])$ and $J_2 = \frac{1}{2}\ad([\theta X, Z])$. We just proved $(J_1 \circ J_2)\rvert_{\g{g}_{\alpha_k}} = -(J_2 \circ J_1) \rvert_{\g{g}_{\alpha_k}}$. Hence, using (\ref{k0:ii}) and defining  $J_3 = J_1 \circ J_2$, the result follows.
\end{proof}

\begin{remark}
	\rm We state here a generalization of Lemma \ref{k0} to arbitrary symmetric spaces of non-compact type. 
	Assume that $\lambda \in \Delta^{+}$ with $2 \lambda \notin \Delta^+$. Then every $2$-dimensional subspace ${\mathbb R}X \oplus {\mathbb R}Y$ of $\g{g}_{\lambda}$, with $X,Y \in \g{g}_{\lambda}$ orthonormal, can be viewed as a complex vector space with complex structure $\frac{1}{2 |\lambda|^2} \ad([\theta X, Y])$. Furthermore, each $4$-dimensional subspace of $\g{g}_{\lambda}$ can be described as a quaternionic subspace. Choose $X, Y, Z \in \g{g}_\lambda$ orthonormal. First, using $\theta \rvert_{\g{k}} = \id_{\g{k}}$ and the Jacobi identity, we deduce
	\begin{align}\label{quaternionic:k0}
	[[\theta X, Y], Z] & =   [[X, \theta Y], Z] = - [[\theta Y, Z], X]  = - [[Y, \theta Z], X] \nonumber \\ &
	= [[\theta Z, X], Y]  =  [[Z, \theta X], Y]  = -[[\theta X, Y], Z],
	\end{align}
	which implies $[[\theta X, Y], Z] = 0$. Let $W$ be a $4$-dimensional subspace of $\g{g}_{\lambda}$ and $X,Y,Z,T \in W$ be orthonormal. Then $J_1, J_2, J_3$ with 
	\begin{align*}
	J_1 & = \frac{1}{2 |\lambda|^2} (\ad([\theta X, Y] +\ad([\theta Z, T])), \\
	J_2 & = \frac{1}{2 |\lambda|^2} (\ad([\theta X, Z] -\ad([\theta Y, T])), \\
	J_3 & = J_1 \circ J_2
	\end{align*}
	is a quaternionic structure on $W$.
\end{remark}

If we think about our symmetric spaces of type $A_2$ in terms of matrices, we have canonical real, complex, quaternionic or octonionic structures on the root spaces. More precisely, the Iwasawa decomposition $G = KAN$ 
gives
\[
G/K = AN =
\left\{
\begin{pmatrix}
x_{11} & x_{12} & x_{13} \\
0 & x_{22} & x_{23} \\
0 & 0 & x_{33}
\end{pmatrix}: 
\begin{array}{l}
x_{11},x_{22},x_{33} \in {\mathbb R};\\ 
x_{12},x_{13},x_{23} \in {\mathbb F};\\ 
x_{11}x_{22}x_{33} = 1
\end{array}\right\}
\]
with 
\[
{\mathbb F} = 
\begin{cases}
{\mathbb R} & \text{ if } G/K = SL_3({\mathbb R})/SO_3, \\
{\mathbb C} & \text{ if } G/K = SL_3({\mathbb C})/SU_3 ,\\
{\mathbb H} & \text{ if } G/K = SL_3({\mathbb H})/Sp_3 ,\\
{\mathbb O} & \text{ if } G/K = E^{-26}_6/F_4. 
\end{cases}
\]
The $x_{12}$- and $x_{23}$-entries correspond (on Lie algebra level) to the root spaces $\g{g}_{\alpha_0}$ and $\g{g}_{\alpha_1}$ respectively, and the $x_{13}$-entry corresponds to the root space $\g{g}_{\alpha_0+\alpha_1}$. The standard examples of CPC submanifolds in these symmetric spaces are given by
\[
\left\{
\begin{pmatrix}
x_{11} & x_{12} & x_{13} \\
0 & x_{22} & x_{23} \\
0 & 0 & x_{33}
\end{pmatrix}
:
\begin{array}{l}
x_{11},x_{22},x_{33} \in {\mathbb R};\\
x_{12},x_{23} \in {\mathbb F} \ominus {\mathbb F}_0; \\
x_{13} \in {\mathbb F};  x_{11}x_{22}x_{33} = 1
\end{array}
\right\}
\]
with ${\mathbb F}_0 \in \{{\mathbb R},{\mathbb C},{\mathbb H},{\mathbb O}\}$ and ${\mathbb F}_0 \subseteq {\mathbb F}$.
If ${\mathbb F}_0 = {\mathbb F}$, we get the totally geodesic submanifolds 
\begin{align*}
{\mathbb R}H^2 \times {\mathbb R} & \subset SL_3({\mathbb R})/SO_3,\\
{\mathbb R}H^3 \times {\mathbb R} & \subset SL_3({\mathbb C})/SU_3, \\
{\mathbb R}H^5 \times {\mathbb R} & \subset SL_3({\mathbb H})/Sp_3, \\
{\mathbb R}H^9 \times {\mathbb R} & \subset E^{-26}_6/F_4.
\end{align*}
In all other cases the submanifold is not totally geodesic. The following 
result makes this more precise.

\begin{theorem}\label{characterisation2}
	Let $\g{s}$ be the subalgebra of $\g{a} \oplus \g{n}$ defined by 
	\[
	\g{s} = \g{a} \oplus (\g{g}_{\alpha_0} \ominus V_0) \oplus (\g{g}_{\alpha_1} \ominus V_1) \oplus \g{g}_{\alpha_0+\alpha_1},
	\]
	$V_0,V_1 \neq \{0\}$, and $S$ be the connected closed subgroup of $AN$ with Lie algebra $\g{s}$. Then $S \cdot o$ is a CPC submanifold if and only if one of the following statements holds:
	\begin{enumerate} [{\rm (i)}]
		\item \label{itm:total} $V_0 \oplus V_1 = \g{g}_{\alpha_0} \oplus \g{g}_{\alpha_1}$;  or
		\item \label{itm:proper}$V_0 \oplus V_1$ is a proper subset of $\g{g}_{\alpha_0} \oplus \g{g}_{\alpha_1}$ and
		\begin{enumerate} [{\rm (a)}] 
			\item \label{itm:real} $V_0$ and $V_1$ are isomorphic to $\mathbb{R}$; 
 or
			\item \label{itm:complex} $V_0$ and $V_1$ are isomorphic to $\mathbb{C}$ and there exists $T \in \g{k}_0$ such that $\ad(T)$ defines complex structures on $V_0$ and $V_1$ and vanishes on $[V_0, V_1]$; or 
			\item  \label{itm:quaternionic} $V_0$ and $V_1$ are isomorphic to $\mathbb{H}$ and there exists a subset $\g{l} \subseteq \g{k}_0$ such that $\ad(\g{l})$ defines quaternionic structures on $V_0$ and $V_1$ and vanishes on $[V_0, V_1]$. 
		\end{enumerate}
	\end{enumerate}
\end{theorem}

\begin{proof}
	Assume that $S \cdot o$ is a CPC submanifold. From Proposition~\ref{characterisation} we have $\dim(V_0) = \dim(V_1) = \dim([V_0,V_1])$. Recall that $T_k = \g{g}_{\alpha_k} \ominus V_k$, with $k \in \{0, 1\}$, and hence $\dim(T_0) = \dim(T_1)$. If $\dim(T_0) \leq 1$, we have (\ref{itm:total}) or (\ref{itm:real}). Assume that $\dim(T_0) \geq 2$. From Lemma~\ref{lemma:auxiliar:characterisation} we get $[V_0, T_1] = [V_1, T_0]$ and $\dim([V_0, T_1]) = \dim(T_0) \geq 2$. Note that $[V_0, T_1] \subseteq \g{g}_{\alpha_0+\alpha_1}$. Thus, using elements in $[V_0, T_1]$, we can construct complex structures (following Lemma \ref{k0}(\ref{k0:ii}) 
if $\dim(T_0) =2$) or quaternionic structures (following Lemma \ref{k0}(\ref{k0:iii}) if $\dim(T_0) > 2$) on $\g{g}_{\alpha_0}$ and $\g{g}_{\alpha_1}$. From (\ref{quaternionic:k0}) we deduce that these structures vanish on $[V_0, V_1]$. Thus it remains to check that these structures can be 
restricted to $V_0$ and $V_1$. In other words, we need to verify that $\langle [[\theta X, Y], \xi_k], Z_k \rangle_{AN} =0$ for $X,Y \in [V_0, T_1] = [V_1, T_0]$, $\xi_k \in V_k$ and $Z_k \in T_k$. Let $j \in \{0,1\}$ with $j \neq k$. There exist $L_j \in T_j$ and $\eta_j \in V_j$ so that $X = \phi_{\xi_k}(L_j)$ and $Y = \phi_{Z_k}(\eta_j)$. Then, using the Jacobi identity, the fact that $\langle \cdot, \cdot \rangle_{B_{\theta}}$ is $\theta$-invariant, (\ref{cartan:inner}) and Proposition~\ref{proposition:main1}, we obtain
	\begin{align*}
	\langle [[\theta X, Y], \xi_k], Z_k \rangle_{B_{\theta}} &  =  - \langle [[\xi_k, \theta X], Y], Z_k \rangle_{B_{\theta}}  =  \langle [\xi_k, 
\theta X], [Z_k, \theta Y] \rangle_{B_{\theta}} \\
	&
	= \langle [\theta \xi_k, X], [\theta Z_k, Y] \rangle_{B_{\theta}} \\ & 
= 2 \langle \phi_{\theta \xi_k} \circ \phi_{\xi_k}(L_j), \phi_{\theta Z_k} \circ  \phi_{Z_k}(\eta_j) \rangle_{B_{\theta}}
	\\ &= 2 \langle L_j, \eta_j \rangle_{B_{\theta}} = 0,
	\end{align*}
	which implies that (\ref{itm:complex}) or (\ref{itm:quaternionic}) holds.

	Conversely, if (\ref{itm:total}) or (\ref{itm:real}) holds, then $S \cdot o$ is a CPC submanifold by Proposition~\ref{characterisation}. For case 
(\ref{itm:complex}), put $J = \ad(K)$ with $K \in \g{k}_0$. By assumption, we can write $V_k = {\mathbb R}X_k \oplus {\mathbb R}JX_k$ with $0 \neq X_k \in V_k$. Then $[V_0,V_1]$ is spanned by $[X_0,X_1]$, $[JX_0,JX_1]$, $[X_0,JX_1]$, $[JX_0,X_1]$. Since $J = \ad(K)$ is a derivation and 
vanishes on $[V_0,V_1]$, we have
	\begin{align*}
	0 & = J[X_0,X_1] = [JX_0,X_1] + [X_0,JX_1],\\
	0 & = J^2[X_0,X_1] = [J^2X_0,X_1] + 2[JX_0,JX_1] + [X_0,J^2X_1]\\
	& = 2([JX_0,JX_1] - [X_0,X_1]),
	\end{align*} 
	which implies $\dim([V_0,V_1]) = 2$. Thus $S \cdot o$ is a CPC submanifold by Proposition~\ref{characterisation}. In case (\ref{itm:quaternionic}) we can write $J_\nu = \ad(K_\nu)$, $K_{\nu} \in \g{k}_0$, $\nu = 1,2,3$,  for the quaternionic structure. Then $V_k$ is spanned by $X_k,J_1X_k,J_2X_k,J_3X_k$ with $0 \neq X_k \in V_k$. As above, we get $[J_\nu X_0,X_1] = - [X_0,J_\nu X_1]$ and 
	$[J_\nu X_0,J_\nu X_1] = [X_0,X_1]$.
	For $\nu,\mu \in \{1,2,3\}$ with $\nu \neq \mu$ we have $J_\nu J_\mu = 
\pm J_\rho$ and hence $[J_\nu X_0,J_\mu X_1] = [J_\nu^2 X_0,J_\nu J_\mu 
X_1] = \pm [X_0,J_\rho X_1]$.
	Altogether this implies that $\dim([V_0,V_1]) = 4$ and from Proposition~\ref{characterisation} we conclude that $S \cdot o$ is a CPC submanifold.
\end{proof}

This finishes the proof of the Main Theorem for the four symmetric spaces 
$SL_{3}(\mathbb{R})/SO_{3}$, $SL_{3}(\mathbb{C})/SU_{3}$, $SL_{3}(\mathbb{H})/Sp_{3}$ and  $E^{-26}_6/F_4$. Recall that this is equivalent to characterize the CPC property of the shape operator $\Ss_{\xi}$ of the examples we constructed in a general symmetric space $G/K$, when it is restricted to the $\Ss_{\xi}$-invariant subspace $\bigoplus_{\gamma \in [\alpha_0]} \g{g}_{\gamma}^{\top}$.

As we mentioned at the beginning of this section, all the examples of the 
Main Theorem can be described as canonical extensions of CPC submanifolds 
in the above four symmetric spaces. As was shown in \cite{miguel}, several geometric properties of submanifolds are preserved via canonical extensions. However, the CPC property is not preserved in general by canonical extension. For example, the maximal flat $A \cdot o \cong {\mathbb E}^2$ is a totally geodesic submanifold of $SL_{3}(\mathbb{R})/SO_{3}$. However, its canonical extension to the symmetric space $SL_{4}(\mathbb{R})/SO_{4}$ is not even austere. For this reason we need to analyze more thoroughly the shape operator of the examples described in the Main Theorem.

\section{Canonical extensions of CPC submanifolds}\label{canonical:extension}

In this section we calculate the shape operator of the canonical extensions of the examples that we investigated in the previous section. We will conclude that these canonical extensions are also CPC submanifolds.

The concept of canonical extensions was introduced in \cite{BT13} and studied in the context of cohomogeneity one actions. We refer the reader to \cite{BT13} for details, but roughly it works as follows. Every subset $\Phi$ of $\Pi$ determines a parabolic subgroup $Q_\Phi$ of $G$. Let $Q_\Phi = M_\Phi A_\Phi N_\Phi$ be its Langlands decomposition. The orbit $B_\Phi = M_\Phi \cdot o$ is a totally geodesic submanifold of $M$ whose rank is equal to the cardinality of $\Phi$. If $S$ is a subgroup of $M_\Phi$, then $SA_\Phi N_\Phi$ is the canonical extension of $S$ from $M_\Phi$ 
to $G$ and the orbit $SA_\Phi N_\Phi \cdot o \subseteq M$ is the canonical extension of the orbit $S \cdot o \subseteq B_\Phi$. If there exist $\alpha_0,\alpha_1 \in \Pi$ so that $\alpha_0$ and $\alpha_1$ are connected in the Dynkin diagram of $M = G/K$ by a single edge, and put $\Phi = \{\alpha_0, \alpha_1\}$, then $B_\Phi$ is one of the symmetric spaces $SL_{3}(\mathbb{R})/SO_{3}$, $SL_{3}(\mathbb{C})/SU_{3}$, $SL_{3}(\mathbb{H})/Sp_{3}$ or $E_{6}^{-26}/F_4$. In Theorem~\ref{characterisation2} we classified the CPC  submanifolds of $B_{\Phi}$ of the form $S \cdot o$, where $\g{s} = \g{a} \oplus ((\g{g}_{\alpha_0} \oplus \g{g}_{\alpha_1}) \ominus V) \oplus \g{g}_{\alpha_0 + \alpha_1}$. In this section we will prove that the canonical extension of $S \cdot o \subset B_{\Phi}$ to the symmetric space $M = G/K$ is a CPC submanifold if and only if $S \cdot o$ is a CPC submanifold of $B_\Phi$.

Let $G/K$ be a symmetric space of non-compact type, with at least two simple roots $\alpha_0$ and $\alpha_1$ connected by a single edge in its Dynkin diagram. Our approach for constructing new examples was to take a subspace $V \subset \g{g}_{\alpha_0} \oplus \g{g}_{\alpha_1}$ and define the 
subalgebra $\g{s} = \g{a} \oplus (\g{n} \ominus V)$. Let $S$ be the connected closed subgroup of $AN$ with Lie algebra $\g{s}$. We are interested in the geometry of the submanifold $S \cdot o$ of $AN = G/K$.

Let $\xi \in V$ be a unit normal vector. As we clarified in Section \ref{construction:examples}, the subspaces $\bigoplus_{\gamma \in [\lambda]} \g{g}_{\gamma}^{\top} $ in the orthogonal decomposition 
\begin{equation*}
\g{n} \ominus V  = \bigoplus_{\lambda \in \Delta^{+}/ \sim} \left(\bigoplus_{\gamma \in [\lambda]} \g{g}_{\gamma}^{\top} \right)
\end{equation*}
are all $\Ss_{\xi}$-invariant. Therefore, $S \cdot o$ is a CPC submanifold of $M$ if and only if for all unit normal vectors $\xi$ the shape operator $\Ss_{\xi}$ has the same eigenvalues when restricted to each of these 
subspaces. We clarified this in Theorem~\ref{characterisation2} for the invariant subspace $\bigoplus_{\gamma \in [\alpha_0]} \g{g}_{\gamma}^\top$. In this section we will clarify this for the remaining subspaces in the above decomposition.
The following result explains the above decomposition in more detail.

\begin{lemma}\label{structure:strings}
	Let $\Delta$ be the root system of a symmetric space of non-compact type 
with at least two simple roots ${\alpha_0}$ and ${\alpha_1}$ connected by 
a single edge in the Dynkin diagram. Then the equivalence class $[\lambda]$ of a positive root  $\lambda \in \Delta^{+} \backslash (\mathbb{R}{\alpha_0} \oplus \mathbb{R}{\alpha_1})$, which has minimum level in its $(\alpha_0,\alpha_1)$-string, can be described as follows (with $k \in \{0,1\}$ and indices modulo $2$): 
	\begin{enumerate}[{\rm (i)}]
		\item $[\lambda] = \{\lambda\}$, if $\langle \lambda, \alpha_0 \rangle = 0 = \langle \lambda, \alpha_1 \rangle$.
		\item $[\lambda] =  \{\lambda, \lambda + {\alpha_k}, \lambda + {\alpha_{k}} + {\alpha_{k+1}}\}$, if $|\alpha_k| \geq |\lambda|$ and $\langle \lambda, \alpha_k \rangle \neq 0$.\label{structure:strings:ii}
	%\item $[\lambda] =  \{\lambda, \lambda + {\alpha_k}, \lambda + {\alpha_k}  + {\alpha_{k+1}}, \lambda + 2{\alpha_k}, \lambda + 2 {\alpha_k} + {\alpha_{k+1}}, \lambda + 2 {\alpha_k} + 2{\alpha_{k+1}} \}$, if $|\alpha_k| < |\lambda|$ and $\langle \lambda, {\alpha_k} \rangle \neq 0$. \label{structure:strings:iii}
\item $[\lambda] =  \{\lambda +\varepsilon_k {\alpha_k}  + \varepsilon_{k+1} {\alpha_{k+1}} \, : \,   \varepsilon_k, \varepsilon_{k+1} \in \{0,1,2\}, \, \varepsilon_k \geq \varepsilon_{k+1}  \}$, if $|\alpha_k| < |\lambda|$ and $\langle \lambda, {\alpha_k} \rangle \neq 0$ \label{structure:strings:iii}.	
\end{enumerate}
\end{lemma}

\begin{proof}
	Since $\lambda,\alpha_0,\alpha_1$ are linearly independent, they generate a root system $R \subseteq \Delta$ of rank $3$. We can assume that $\lambda,\alpha_0,\alpha_1$ are positive roots in $R$. 
	
	If $R$ is reducible, we must have $R \cong A_2 \oplus A_1$ with $A_2$ generated by $\alpha_0$ and $\alpha_1$ and $A_1$ generated by $\lambda$. It 
is clear that this is equivalent to $[\lambda] = \{\lambda\}$ and $\langle \lambda, \alpha_k \rangle = 0$ for $k \in \{0,1\}$, which corresponds to case (i). 
	
	If $R$ is irreducible, then $R$ is isomorphic to $A_3$, $B_3$, $C_3$ or $BC_3$. The result follows by inspecting these rank $3$ root systems case 
by case and taking into account that $\lambda$ has minimum level in its $(\alpha_0,\alpha_1)$-string. If $R \cong A_3$ or $R \cong B_3$, we get (ii). If $R \cong C_3$, we get (iii). Finally, if $R \cong BC_3$, then $\lambda$ is either reduced or non-reduced. If $\lambda$ is reduced, we get (ii), and if $\lambda$ is non-reduced, we get (iii).
\end{proof}
In view of Lemma~\ref{structure:strings} we have to investigate three cases. 

Case (i): $[\lambda] = \{\lambda\}$. From (\ref{shape}) and (\ref{bracket:relation}) we see that $\Ss_{\xi}$ vanishes on $\g{g}_\lambda = \g{g}_\lambda^\top$. 

Case (ii): $[\lambda] =  \{\lambda, \lambda + {\alpha_k}, \lambda + {\alpha_{k}} + {\alpha_{k+1}}\}$. We consider the subspace
\[
\g{g}_{ \lambda} \oplus \g{g}_{ \lambda + \alpha_k} \oplus \g{g}_{\lambda 
+ \alpha_k + \alpha_{k+1}} \subseteq \g{s}.
\]
We write $\xi = \cos(\varphi) \xi_k + \sin(\varphi) \xi_{k+1}$ for $\varphi \in [0, \frac{\pi}{2}]$, $\xi_k \in V_k$ and $\xi_{k+1} \in V_{k+1}$. Note that $A_{\alpha_k, \lambda} = -1$ and $A_{\alpha_{k+1}, \lambda + \alpha_k} = -1$. For the pairs $(\gamma, \nu) = (\lambda, \alpha_k)$ and $(\gamma, \nu) = (\lambda + \alpha_k, \alpha_{k+1})$ we obtain from Proposition~\ref{proposition:main1} that $\g{g}_{ \lambda + \alpha_k} = \phi_{\xi_k}(\g{g}_{\lambda})$ and $\g{g}_{\lambda + \alpha_k + \alpha_{k+1}} = (\phi_{\xi_{k+1}} \circ \phi_{\xi_k})(\g{g}_{\lambda})$. Let $0 \neq X_{\lambda} \in \g{g}_{\lambda}$. From (\ref{shape}) and (\ref{bracket:relation}), together with the fact that $\lambda + \alpha_{k+1} \notin \Delta$, we get
\[
\Ss_{\xi_{k+1}} X_{\lambda} = \Ss_{\xi_k}(\phi_{\xi_{k+1}} \circ \phi_{\xi_k})(X_{\lambda})=0.
\]
For the pair $(\gamma, \nu) \in \{(\lambda, \alpha_k),(\lambda + \alpha_k, \alpha_{k+1})\}$, we deduce from (\ref{definition:shape:operator}) and Proposition~\ref{proposition:main1} that
\begin{align*}
\Ss_{\xi_{k}} X_{\lambda} & = -\left(\nabla_{X_{\lambda}} \xi_{k} \right)^{\top} = \frac{|\alpha_0|}{2} \phi_{\xi_k}(X_{\lambda}), \\
\Ss_{\xi_{k}} \phi_{\xi_k}(X_{\lambda}) & = -\left(\nabla_{\phi_{\xi_k}(X_{\lambda})} \xi_{k} \right)^{\top} = \frac{|\alpha_0|}{2} X_{\lambda}, \\
\Ss_{\xi_{k+1}} \phi_{\xi_k}(X_{\lambda}) & = -\left(\nabla_{\phi_{\xi_k}(X_{\lambda})} \xi_{k+1} \right)^{\top}  = \frac{|\alpha_0|}{2} (\phi_{\xi_{k+1}} \circ \phi_{\xi_k})(X_{\lambda}),\\
\Ss_{\xi_{k+1}} (\phi_{\xi_{k+1}} \circ \phi_{\xi_k})(X_{\lambda}) & = -\left(\nabla_{(\phi_{\xi_{k+1}} \circ \phi_{\xi_k})(X_{\lambda})} \xi_{k+1} \right)^{\top}  = \frac{|\alpha_0|}{2} \phi_{\xi_k}(X_{\lambda}).
\end{align*}
Thus, the $3$-dimensional vector space spanned by the vectors $X$, $\phi_{\xi_k}(X)$ and $(\phi_{\xi_{k+1}} \circ \phi_{\xi_k})(X)$ is $\Ss_\xi$-invariant. It follows that the matrix representation of $\Ss_{\xi}$ is given by $\dim (\g{g}_{\lambda})$ blocks
\[
\frac{|{\alpha_0}|}{2} 
\left(\begin{array}{ccc} 0& \cos(\varphi) & 0 \\   \cos(\varphi) & 0 & \sin(\varphi)  \\  0& \sin(\varphi) & 0  
\end{array}\right)
\]
with respect to the decomposition $\g{g}_{ \lambda} \oplus \phi_{\xi_k}(\g{g}_{\lambda}) \oplus (\phi_{\xi_{k+1}} \circ \phi_{\xi_k})(\g{g}_{\lambda})$. An elementary calculation shows that $\Ss_\xi$ restricted to $\g{g}_{ \lambda} \oplus \g{g}_{ \lambda + \alpha_k} \oplus \g{g}_{\lambda + \alpha_k + \alpha_{k+1}}$ has the three eigenvalues $0$ and $\pm \frac{|\alpha_0|}{2}$, all of them with multiplicity $\dim(\g{g}_\lambda)$. Thus we established that the eigenvalues of $\Ss_\xi$ are independent of the choice of $\xi$ for case (ii). Note that cases (i) and (ii) together already settle the problem if $G/K$ is a symmetric space whose Dynkin diagram is of type $A_r$, $B_r$, $D_r$, $E_6$, $E_7$ or $E_8$. 

Case (iii): $[\lambda] =  \{\lambda, \lambda + {\alpha_k}, \lambda + {\alpha_k} + {\alpha_{k+1}}, \lambda + 2{\alpha_k}, \lambda + 2 {\alpha_k} + {\alpha_{k+1}}, \lambda + 2 {\alpha_k} + 2{\alpha_{k+1}} \}$. We consider the subspace 
\[
\g{g}_{\lambda} \oplus \g{g}_{\lambda + \alpha_k} \oplus \g{g}_{\lambda + 
2\alpha_k} \oplus \g{g}_{\lambda + \alpha_k + \alpha_{k+1}} \oplus \g{g}_{\lambda + 2 \alpha_k+ \alpha_{k+1}} \oplus \g{g}_{\lambda + 2 \alpha_k + 
2\alpha_{k+1}} \subseteq \g{s}.
\]
We need to understand better the behavior of the Levi-Civita connection when restricted to this subspace. As we did in Proposition \ref{proposition:main1}, we will calculate the Levi-Civita connection using the map $\phi_{\xi}$ defined in (\ref{definition:phi}).

\begin{proposition}\label{proposition:main2}
	Let $\gamma \in \Delta^{+}$ be the root of minimum level in its $\nu$-string, for $\nu \in \Delta^{+}$ non-proportional to $\gamma$ satisfying $A_{\nu, \gamma} = -2$. Let $\xi \in \g{g}_{\nu}$ be a unit vector with respect to $\langle \cdot, \cdot \rangle_{AN}$ and $X \in \g{g}_{\gamma}$. 
Then:
	\begin{itemize} 
		\item[(i)] $\nabla_X \xi = - \frac{|\nu|}{\sqrt{2}} {\phi_{\xi}(X)}$;
		\item[(ii)] $\nabla_{\phi_{\xi}(X)} \xi = - \frac{|\nu|}{\sqrt{2}} (X+\phi_{\xi}^{2}(X))$;
		\item[(iii)] $\nabla_{\phi_{\xi}^{2} (X)} \xi =- \frac{|\nu|}{\sqrt{2}}\phi_{\xi}(X)$;
		\item[(iv)] $\phi_{\xi}^{2} \rvert_{\g{g}_{\gamma}} \colon \g{g}_{\gamma} \to \g{g}_{\gamma+2\nu}$ is a linear isometry;
		\item[(v)] $\nabla_{W} \xi = 0$ for all $W \in \g{g}_{\gamma + \nu} \ominus \phi_{\xi}(\g{g}_{\gamma})$.
	\end{itemize}	
\end{proposition}

\begin{proof}
	Using (\ref{levi:connection:examples}) and (\ref{definition:phi}) we easily obtain $\nabla_X \xi = - \frac{|\nu|}{\sqrt{2}} {\phi_{\xi}(X)}$. The same arguments together with Lemma~\ref{lemma:phi:injective} show that
	\[
	\nabla_{\phi_{\xi}(X)} \xi = \frac{1}{2}([\phi_{\xi}(X), \xi] -[\phi_{\xi}(X), \theta \xi]) = - \frac{|\nu|}{\sqrt{2}}\left(\phi_{\xi}^{2}(X)+X\right).
	\]
	Note that $A_{\nu, \gamma + \nu} = 0$. Thus, combining (\ref{levi:connection:examples}), (\ref{definition:phi}) and the fact that $\gamma + 3\nu$ is not a root with Lemma~\ref{lemma:ad}(\ref{lemma:ad:iii}), we obtain
	\[
	\nabla_{\phi_{\xi}^{2} (X)} \xi = -\frac{1}{2}[\phi_{\xi}^{2} (X), \theta \xi]    =  \frac{1}{4|\nu|^{2}}  [\theta \xi, [\xi, [\xi, X]]] = -\frac{|\nu|}{\sqrt{2}} \phi_{\xi}(X).
	\]
	Moreover, using again Lemma~\ref{lemma:ad}(\ref{lemma:ad:iii}), we deduce
	\begin{align*}
	\langle \phi_{\xi}^{2}(Y), \phi_{\xi}^{2}(Z) \rangle_{AN} & = \frac{1}{4 |\nu|^{4} }\langle [\xi,[\xi, Y]], [\xi,[\xi, Z]] \rangle_{AN} \\ &= 
-\frac{1}{4 |\nu|^{4} }\langle [\xi, Y], [\theta \xi, [\xi,[\xi, Z]] ] \rangle_{AN} \\
	&   = \frac{1}{2 |\nu|^{2} } \langle [\xi, Y], [\xi, Z] \rangle_{AN}  = \langle \phi_{\xi}(Y), \phi_{\xi}(Z) \rangle_{AN} =\langle Y, Z \rangle_{AN}
	\end{align*}
	for $Y,Z \in \g{g}_{\gamma}$. It is then clear that $\phi_{\xi}^{2}$ is an injective linear map preserving the inner product when restricted to $\g{g}_{\gamma}$. Furthermore, from Lemma \ref{root:spaces:dimension} we know that $\dim(\g{g}_{\gamma}) = \dim(\g{g}_{\gamma + 2\nu})$, and thus 
$\phi_{\xi}^{2} \rvert_{\g{g}_{\gamma}} \colon \g{g}_{\gamma} \to \g{g}_{\gamma+2\nu} $ is a linear isometry. Note that Lemma~\ref{lemma:ad}(\ref{lemma:ad:iii}) for $A_{\nu, \gamma} = -2$ is equivalent to $(\phi_{\theta \xi} \circ \phi^{2}_{\xi})\rvert_{\g{g}_{\gamma}}= \phi_{\xi} \rvert_{\g{g}_{\gamma}}$. Then, we deduce that $\phi_{\xi}(\g{g}_{\gamma}) = \phi_{\theta \xi}(\g{g}_{\gamma+ 2 \nu})$. To complete the proof, fix a vector $W \in \g{g}_{\gamma+\nu} \ominus \phi_{\xi}(\g{g}_{\gamma}) = \g{g}_{\gamma+\nu} \ominus \phi_{\theta \xi}(\g{g}_{\gamma + 2 \nu}) $. On the one hand, we have $\langle \phi_{\xi} (W), Y \rangle_{AN} = \langle 
W, \phi_{\theta \xi}(Y) \rangle_{AN} = 0$ for all $Y \in \g{g}_{\gamma+2\nu}$. On the other hand, $\langle \phi_{\theta \xi} (W), Z \rangle_{AN} 
= \langle W, \phi_{\xi}(Z) \rangle_{AN} = 0$ for all $Z \in \g{g}_{\gamma}$. This implies $\nabla_{W} \xi =0$ for all $W \in \g{g}_{\gamma+\nu} \ominus \phi_{\xi}(\g{g}_{\gamma}) $, which finishes the proof.
\end{proof}

Let $\xi \in V$ be a unit vector and, as above, put $\xi = \cos(\varphi) \xi_k + \sin(\varphi) \xi_{k+1}$. We first study the shape operator $\Ss_{\xi}$ on the subspace 
\begin{equation}\label{structure:strings:decomposition}
\begin{aligned}
&\g{g}_{\lambda} \oplus \phi_{\xi_{k}}(\g{g}_{\lambda}) \oplus \phi^2_{\xi_{k}}(\g{g}_{\lambda}) \oplus (\phi_{\xi_{k+1}}\circ \phi_{\xi_k}) (\g{g}_{\lambda})\\
& \oplus (\phi_{\xi_{k+1}}\circ \phi^{2}_{\xi_k}) (\g{g}_{\lambda}) \oplus (\phi^{2}_{\xi_{k+1}}\circ \phi^{2}_{\xi_k})(\g{g}_{\lambda}).
\end{aligned}
\end{equation}
Let  $0 \neq X_{\lambda} \in\g{g}_{\lambda}$. First, using (\ref{shape}), 
(\ref{bracket:relation}) and the fact that neither $\lambda + \alpha_{k+1}$ nor $\lambda + 2\alpha_{k+1} + \alpha_k$ are roots, we deduce
\[
\Ss_{\xi_{k+1}}X_{\lambda} = \Ss_{\xi_k} (\phi^{2}_{\xi_{k+1}} \circ \phi^{2}_{\xi_k})(X_{\lambda}) = 0.
\]
We will analyze the $\alpha_k$-string of $\lambda$ and the $\alpha_{k+1}$-string of $\lambda + 2 \alpha_k$ simultaneously. Let $\mu \in \{\lambda, 
\lambda + 2 \alpha_k\}$ and define $r(\mu) = k$ if $\mu = \lambda$ and $r(\mu) = k+1$ otherwise. Put $X_{\mu} = X_{\lambda}$ if $\mu = \lambda$ and $X_{\mu} = \phi^{2}_{\xi_k}(X_{\lambda})$ otherwise. Using (\ref{definition:shape:operator}) and Proposition~\ref{proposition:main2} 
we obtain
\begin{align*}
\Ss_{\xi_{r(\mu)}} X_{\mu}  & = - \left( \nabla_{X_{\mu}} \xi_{r(\mu)} \right)^{\top} = \frac{|\alpha_0|}{\sqrt{2}} \phi_{\xi_{r(\mu)}}(X_{\mu}),\\
\Ss_{\xi_{r(\mu)}} \phi_{\xi_{r(\mu)}}(X_{\mu})  & = - \left( \nabla_{\phi_{\xi_{r(\mu)}}(X_{\mu})} \xi_{r(\mu)} \right)^{\top} =\frac{|\alpha_0|}{\sqrt{2}} (X_{\mu} + \phi^{2}_{\xi_{r(\mu)}}(X_{\mu})), \\
\Ss_{\xi_{r(\mu)}} \phi^{2}_{\xi_{r(\mu)}}(X_{\mu})  & = - \left( \nabla_{\phi^{2}_{\xi_{r(\mu)}}(X_{\mu})} \xi_{r(\mu)} \right)^{\top} = \frac{|\alpha_0|}{\sqrt{2}} \phi_{\xi_{r(\mu)}}(X_{\mu}).
\end{align*}
Note that $A_{\alpha_{k+1}, \lambda + \alpha_k} = A_{\alpha_k, \lambda + \alpha_k +\alpha_{k+1}} = -1$. Then, using (\ref{definition:shape:operator}) and Proposition~\ref{proposition:main1} for the pair $(\gamma, \nu) \in \{(\lambda + \alpha_k, \alpha_{k+1}), (\lambda + \alpha_k + \alpha_{k+1}, \alpha_k)\}$, we get
\begin{align*}
\Ss_{\xi_{k+1}} \phi_{\xi_k}(X_{\lambda}) & = \frac{|\alpha_0|}{2}(\phi_{\xi_{k+1}} \circ \phi_{\xi_k})(X_{\lambda}), \\
\Ss_{\xi_{k+1}} (\phi_{\xi_{k+1}} \circ \phi_{\xi_k})(X_{\lambda}) & = \frac{|\alpha_0|}{2}\phi_{\xi_k}(X_{\lambda}), \\
\Ss_{\xi_k} (\phi_{\xi_{k+1}} \circ \phi_{\xi_k})(X_{\lambda}) & = \frac{|\alpha_0|}{2}(\phi_{\xi_k} \circ \phi_{\xi_{k+1}} \circ \phi_{\xi_k})(X_{\lambda}), 
\end{align*}
\begin{align*}
\Ss_{\xi_k} (\phi_{\xi_k} \circ \phi_{\xi_{k+1}} \circ \phi_{\xi_k})(X_{\lambda}) & = \frac{|\alpha_0|}{2}(\phi_{\xi_{k+1}} \circ \phi_{\xi_k})(X_{\lambda}).
\end{align*}

So far we calculated the shape operator $\Ss_{\xi}$ on the subspace in (\ref{structure:strings:decomposition}). However, all this information is not conclusive as $(\phi_{\xi_k} \circ \phi_{\xi_{k+1}} \circ \phi_{\xi_k})(X_{\lambda})$ and $(\phi_{\xi_{k+1}} \circ \phi_{\xi_k}^{2})(X_{\lambda})$ both belong to  $\g{g}_{\lambda+2\alpha_k + \alpha_{k+1}}$, but we do 
not know how they are related. Consider the $(\alpha_k, \alpha_{k+1})$-string containing $\lambda$:
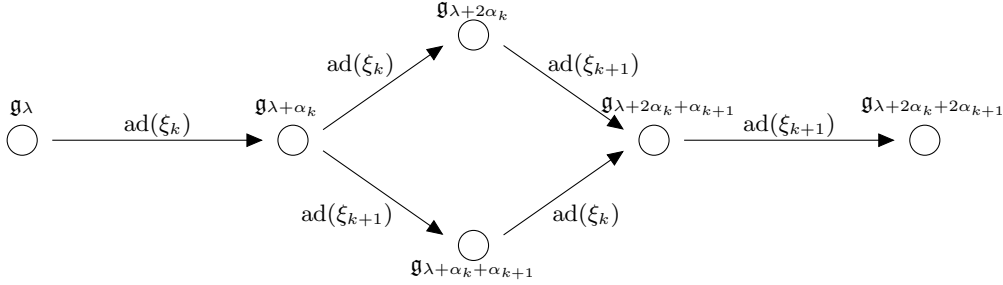
\begin{equation*}
\begin{tikzpicture}[scale=0.35]
\draw(-15.0,0.) circle (0.5cm);
\draw(-6.,0.) circle (0.5cm);
\draw(6,0.) circle (0.5cm);
\draw(0,3.5) circle (0.5cm);
\draw(0,-3.5) circle (0.5cm);
\draw(15.,0.) circle (0.5cm);
\draw[-  triangle 45] (7,0.) -- (14,0);
\draw[-  triangle 45] (-14,0.0)-- (-7,0.0);
\draw[-  triangle 45] (-5,0.35)-- (-1,3.15);
\draw[triangle 45 -] (5,0.35) -- (1,3.15);
\draw[-  triangle 45] (-5,-0.35)-- (-1,-3.15);
\draw[triangle 45 -] (5,-0.35)-- (1,-3.15);
\begin{scriptsize}
\draw[color=black] (-10.5, 0.5) node {$\ad(\xi_k)$};
\draw[color=black] (10.5, 0.5) node {$\ad(\xi_{k+1})$};
\draw[color=black] (-3.75, 2.5) node {$\ad(\xi_k)$};
\draw[color=black] (-4.2, -2.5) node {$\ad(\xi_{k+1})$};
\draw[color=black] (4, 2.5) node {$\ad(\xi_{k+1})$};
\draw[color=black] (3.75, -2.5) node {$\ad(\xi_k)$};
\draw[color=black] (-15, 1.1) node {$\g{g}_{\lambda}$};
\draw[color=black] (-6.2, 1.3) node {$\g{g}_{\lambda+\alpha_k}$};
\draw[color=black] (6.5, 1.3) node {$\g{g}_{\lambda+2\alpha_k + \alpha_{k+1}}$};
\draw[color=black] (0,4.3) node {$\g{g}_{\lambda+2\alpha_k}$};
\draw[color=black] (0,-4.3) node {$\g{g}_{\lambda +  \alpha_k + \alpha_{k+1}}$};
\draw[color=black] (15.3, 1.1) node {$\g{g}_{\lambda +2\alpha_k + 2 \alpha_{k+1}}$};
\end{scriptsize}
\end{tikzpicture}
\end{equation*}
(Note that the nodes in this diagram represent root spaces and not roots.) The problem is that it is not clear whether or not the square diagram in the middle is commutative. More precisely, we do not yet understand the 
behavior of the vector $\phi_{\xi_k}(X_{\lambda})$ depending on the part of the diagram it follows. In terms of brackets, the key point is to understand the relation between $[[\phi_{\xi_k}(X_{\lambda}), \xi_{k}], \xi_{k+1}]$ and  $[[\phi_{\xi_k}(X_{\lambda}), \xi_{k+1}], \xi_{k}]$. Using  (\ref{definition:phi}) and the Jacobi identity twice, we obtain
\begin{align*}
\sqrt{2}|\alpha_0|[\phi_{\xi_k}(X_{\lambda}), [\xi_{k+1}, \xi_{k}]] & = 
- [[X_{\lambda}, \xi_{k}],[\xi_{k+1}, \xi_{k}]] \\
& =[[\xi_k,[\xi_{k+1}, \xi_{k}]],X_{\lambda}] + [[[\xi_{k+1}, \xi_{k}], 
X_{\lambda}], \xi_k] \\& 
=[[[\xi_{k+1}, \xi_{k}], X_{\lambda}], \xi_k] \\
& = -[[[X_{\lambda}, \xi_{k+1}], \xi_k], \xi_k]- [[[\xi_{k}, X_{\lambda}], \xi_{k+1}], \xi_{k}] \\
& = - [[[\xi_{k}, X_{\lambda}], \xi_{k+1}], \xi_{k}] \\ &= -\sqrt{2}|\alpha_0|[[\phi_{\xi_k}(X_{\lambda}), \xi_{k+1}], \xi_k].
\end{align*}
Using the last equality and writing $Y = \phi_{\xi_k}(X_{\lambda})$ for 
the sake of simplicity, we deduce
\begin{align*}
2|\alpha_0|^2(\phi_{\xi_{k+1}} \circ \phi_{\xi_k}) (Y) & =  [\xi_{k+1}, 
[\xi_k, Y]] = - ([\xi_{k}, [Y, \xi_{k+1}]] + [Y, [\xi_{k+1},\xi_{k}]]) \\&=  [[Y, \xi_{k+1}], \xi_{k}] - [Y, [\xi_{k+1},\xi_{k}]] \\
& = [[Y, \xi_{k+1}], \xi_{k}]+[[Y, \xi_{k+1}], \xi_k]  =  2[\xi_k, [\xi_{k+1}, Y]] \\
& = 4|\alpha_0|^2(\phi_{\xi_{k}} \circ \phi_{\xi_{k+1}}) (Y),
\end{align*}
which proves that the diagram is commutative up to a constant. In particular, we established that the vector space spanned by the vectors
\begin{align*}
& X_{\lambda},\ \phi_{\xi_{k}}(X_{\lambda}),\ \phi^2_{\xi_{k}}(X_{\lambda}),\ (\phi_{\xi_{k+1}}\circ \phi_{\xi_k}) (X_{\lambda}),\\ &(\phi_{\xi_{k+1}}\circ \phi^{2}_{\xi_k}) (X_{\lambda}),\ (\phi^{2}_{\xi_{k+1}}\circ \phi^{2}_{\xi_k})(X_{\lambda})
\end{align*}
is $\Ss_{\xi}$-invariant. Therefore, the matrix representation of the shape operator $\Ss_{\xi}$ on that subspace is given by $\dim(\g{g}_{\lambda})$ blocks of the form
\[
\frac{|\alpha_0|}{2}\left(
\begin{array}{@{}c@{}c@{}c@{\ }c@{}c@{}c@{}}
0 & \sqrt{2}\cos (\varphi ) & 0 & 0 & 0 & 0 \\
\sqrt{2}\cos (\varphi ) & 0 & \sqrt{2}\cos (\varphi ) & \sin (\varphi ) & 
0 & 0 \\
0 & \sqrt{2}\cos (\varphi ) & 0 & 0 & \sqrt{2}\sin (\varphi ) & 0 \\
0 & \sin (\varphi ) & 0 & 0 & \cos (\varphi ) & 0 \\
0 & 0 & \sqrt{2}\sin (\varphi ) & \cos (\varphi ) & 0 & \sqrt{2}\sin (\varphi ) \\
0 & 0 & 0 & 0 & \sqrt{2}\sin (\varphi ) & 0 \\
\end{array}
\right)
\]
with respect to the decomposition in (\ref{structure:strings:decomposition}). A straightforward calculation shows that the eigenvalues of $\Ss_{\xi}$ are $\pm |{\alpha_0}|, \pm \frac{|{\alpha_0}|}{2} ,0 $, each of them with multiplicity $\dim(\g{g}_{\lambda})$, except $0$, which has multiplicity $2 \dim(\g{g}_{\lambda})$.

Finally, from Lemma~\ref{root:spaces:dimension} and Lemma~\ref{lemma:ad} we see that
\begin{align*}
& \dim(\g{g}_{\lambda}) = \dim(\g{g}_{\lambda + 2 \alpha_k}) = \dim(\g{g}_{\lambda + 2 \alpha_k + 2 \alpha_{k+1}}) \\
& \leq \dim(\g{g}_{\lambda + \alpha_k}) = \dim(\g{g}_{\lambda +  \alpha_k +  \alpha_{k+1}}) = \dim(\g{g}_{\lambda + 2 \alpha_k +  \alpha_{k+1}}),
\end{align*}
where indices are modulo 2. Define $U = \g{g}_{\lambda + \alpha_k} \ominus \phi_{\xi_k} (\g{g}_{\lambda})$. We still need to analyze the behavior of $\Ss_{\xi}$ on the vector space
\[
U \oplus \phi_{\xi_{k+1}}(U) \oplus (\phi_{\xi_k} \circ \phi_{\xi_{k+1}})(U).
\]
Let  $0 \neq X \in U$. On the one hand, using (\ref{definition:shape:operator}) and Proposition~\ref{proposition:main2}, we obtain
\[
\Ss_{\xi_k} X = \Ss_{\xi_{k+1}} (\phi_{\xi_k} \circ \phi_{\xi_{k+1}}) (X)= 0.
\]
Note that $A_{\alpha_{k+1}, \lambda + \alpha_k} = -1$ and $A_{\alpha_k, 
\lambda + \alpha_k + \alpha_{k+1}} = -1$. On the other hand, using (\ref{definition:shape:operator}) and Proposition~\ref{proposition:main1} for 
the pair $(\lambda, \nu) \in \{(\lambda + \alpha_k, \alpha_{k+1}), ((\lambda + \alpha_k + \alpha_{k+1}, \alpha_k))\}$  we obtain
\begin{align*}
\Ss_{\xi}(X) & = \frac{|\alpha_0|}{2} \sin (\varphi) \phi_{\xi_{k+1}}(X), \\
\Ss_{\xi}\phi_{\xi_{k+1}}(X) & = \frac{|\alpha_0|}{2} (\sin (\varphi) X 
+\cos (\varphi)(\phi_{\xi_k} \circ \phi_{\xi_{k+1}})(X)), \\
\Ss_{\xi}(\phi_{\xi_k} \circ \phi_{\xi_{k+1}})(X) & = \frac{|\alpha_0|}{2} \cos(\varphi) \phi_{\xi_{k+1}}(X).
\end{align*}
Since the vector space generated by the vectors  $X, \phi_{\xi_{k+1}}(X),(\phi_{\xi_k} \circ \phi_{\xi_{k+1}})(X)$ is $\Ss_{\xi}$-invariant, the matrix representation of $\Ss_{\xi}$ on $U \oplus \phi_{\xi_{k+1}}(U) \oplus (\phi_{\xi_k} \circ \phi_{\xi_{k+1}})(U)$ is given by $(\dim(\g{g}_{\lambda + \alpha_k}) - \dim(\g{g}_{\lambda}))$ blocks of the form
\[
\frac{|{\alpha_0}|}{2} 
\left(\begin{array}{ccc} 0& \sin(\varphi) & 0 \\   \sin(\varphi) & 0 & \cos(\varphi)  \\  0& \cos(\varphi) & 0  
\end{array}\right).
\]
The eigenvalues are $0$, $|\alpha_0|/2$ and $- |\alpha_0|/2$, each of them with multiplicity $\dim(\g{g}_{\lambda + \alpha_k}) -\dim(\g{g}_{\lambda})$. Altogether we have now established that the canonical extensions are also CPC submanifolds.

\section{The classification}\label{classification}

In this section we finish the classification in the Main Theorem. We will 
show that if $S \cdot o$ is a CPC submanifold of $M = G/K$, then it must be one of the examples presented in the Main Theorem. More precisely, we will prove that if $S \cdot o$ is a CPC submanifold, then either $V \subseteq \g{g}_{\alpha}$ for some $\alpha \in \Pi'$ or there exist $\alpha_0,\alpha_1 \in \Pi'$ with $A_{\alpha_0, \alpha_1} =A_{\alpha_1, \alpha_0} = -1$ and $V \subseteq \g{g}_{\alpha_0} \oplus \g{g}_{\alpha_1}$. Together with Theorem~\ref{characterisation2} this finishes the classification part of the Main Theorem. We start with a result about the principal curvatures of the submanifold $S \cdot o$. Recall that, according to (\ref{V:decomposition}), we can write $V = \bigoplus_{\alpha \in \psi} V_{\alpha}$, where $V_{\alpha}$ is a non-trivial subspace of $\g{g}_{\alpha}$ 
for each $\alpha \in \psi$. 

\begin{proposition}\label{proposition:principalcurvatures}
	Let $\g{s} = \g{a} \oplus (\g{n} \ominus V)$ be a subalgebra of $\g{a} 
\oplus \g{n}$ with $V = \oplus_{\alpha \in \psi}\ V_{\alpha}$ and $\psi 
\subseteq \Pi'$. Let $\gamma \in \Delta^{+}$ be the root of minimum level 
in its $\nu$-string, for $\nu \in \psi$ non-proportional to $\gamma$. Let 
$I$ be the set of roots in the $\nu$-string of $\gamma$. Consider the restriction of the shape operator $\Ss_{\xi}$ of $S \cdot o$ to the vector space $\bigoplus_{\alpha \in I} \g{g}_{\alpha}^{\top}$, where $\xi$ is a unit vector in $V_{\nu}$.
	\begin{enumerate}[{\rm (i)}]
		\item If $A_{\nu,\gamma} = -1$, then $\pm \frac{|\nu|}{2}$ are principal curvatures and both with multiplicity $\dim(\g{g}_{\gamma}^{\top})$. \label{proposition:principalcurvatures:i}
		\item If $A_{\nu,\gamma} = -2$, then $\pm |\nu|$ are principal curvatures and both with multiplicity $\dim(\g{g}_{\gamma}^{\top})$, and $\pm \frac{|\nu|}{\sqrt{2}}$ are principal curvatures, both with multiplicity $\dim(V_{\gamma})$. \label{proposition:principalcurvatures:ii}
	\end{enumerate}
\end{proposition}

\begin{proof}
	Assume first that $A_{\nu, \gamma} = -1$. In this case the $\nu$-string of $\gamma$ consists of $\gamma,\gamma + \nu$. Since $\gamma + \nu \notin \Pi$, we have $\g{g}_{\gamma + \nu}^{\top} = \g{g}_{\gamma + \nu}$. Let $\xi \in V_{\nu}$ be a unit vector and consider the restriction of the shape operator $\Ss_{\xi}$ to $\g{g}_{\gamma}^{\top} \oplus \g{g}_{\gamma + \nu}$. From (\ref{definition:shape:operator}) and Proposition \ref{proposition:main1} we get
	\begin{equation*}
	\begin{aligned}
	\Ss_{\xi} X  &{}= -(\nabla_{X} \xi)^{\top} = \frac{|\nu|}{2} \phi_{\xi}(X) , \\
	\Ss_{\xi} \phi_{\xi}(X) &{} = -(\nabla_{\phi_{\xi}(X)} \xi)^{\top} = 
\frac{|\nu|}{2} X 
	\end{aligned}
	\end{equation*}
	for $X \in \g{g}_{\gamma}^{\top}$. Then the $2$-dimensional vector space 
spanned by $X, \phi_{\xi}(X)$ is $\Ss_{\xi}$-invariant for all $0 \neq X \in \g{g}_{\gamma}^{\top}$ and all unit vectors $\xi \in V_{\nu}$. Thus the matrix representation of $\Ss_{\xi}$ on $\g{g}_{\gamma}^{\top} \oplus \phi_{\xi} (\g{g}_{\gamma}^{\top})$ consists of $\dim(\g{g}_{\gamma}^{\top})$ blocks of the form
	\[
	\frac{|\nu|}{2} 
	\left(\begin{array}{cc} 0 & 1  \\   1 & 0   
	\end{array}\right).
	\]
	Finally, let $Y \in \phi_{\xi}(V_{\gamma})$ and write $Y = \phi_{\xi}(\eta)$ with $\eta \in V_{\gamma}$. From (\ref{definition:shape:operator}) 
and Proposition~\ref{proposition:main1} we obtain 
	\[
	\Ss_{\xi} Y = \Ss_{\xi} \phi_{\xi}(\eta) = - (\nabla_{\phi_{\xi}(\eta)} \xi)^\top = \frac{|\nu|}{2} \eta^{\top} = 0.
	\]
	Therefore, $\pm \frac{|\nu|}{2}$ are the non-zero principal curvatures of $\Ss_{\xi}$ on  $\g{g}_{\gamma}^{\top} \oplus \g{g}_{\gamma + \nu}$, and both have multiplicity $\dim(\g{g}_{\gamma}^{\top})$. This proves (\ref{proposition:principalcurvatures:i}).
	
	Now assume that $A_{\nu, \gamma} = -2$. Then the $\nu$-string of $\gamma$ consists of $\gamma,\gamma + \nu,\gamma + 2\nu$. Since $\gamma + \nu$ 
and $\gamma + 2\nu$ are not simple roots, we have $\g{g}_{\gamma+\nu}^{\top} = \g{g}_{\gamma+\nu}$ and $\g{g}_{\gamma+2\nu}^{\top} = \g{g}_{\gamma+2\nu}$. Let $\xi$ be a unit vector in $V_{\nu}$ and consider the restriction of the shape operator $\Ss_{\xi}$ to $\g{g}_{\gamma}^{\top} \oplus \g{g}_{\gamma+\nu} \oplus \g{g}_{\gamma+2\nu}$. Let $X \in \g{g}_{\gamma}^{\top}$. From (\ref{definition:shape:operator}) and Proposition~\ref{proposition:main2} we obtain
	\begin{align*}
	\Ss_{\xi} X & = -\left(\nabla_{X} \xi\right)^{\top} = \frac{|\nu|}{\sqrt{2}} \phi_{\xi} (X), \\
	\Ss_{\xi} \phi_{\xi} (X) & = - \left( \nabla_{\phi_{\xi} (X)} \xi \right)^{\top} = \frac{|\nu|}{\sqrt{2}} (\phi^{2}_{\xi} (X) + X)^{\top} = 
\frac{|\nu|}{\sqrt{2}} (\phi^{2}_{\xi} (X) + X), \\
	\Ss_{\xi} \phi^{2}_{\xi}(X) & = - \left(\nabla_{\phi^{2}_{\xi} (X)} \xi \right) = \frac{|\nu|}{\sqrt{2}} \phi_{\xi} (X).
	\end{align*}
	Thus the $3$-dimensional vector space spanned by $X, \phi_{\xi}(X), \phi_{\xi}^{2}(X)$ is $\Ss_{\xi}$-invariant for all $0 \neq X \in \g{g}_{\gamma}^{\top}$ and all unit vectors $\xi \in V_{\nu}$. Thus the matrix representation of $\Ss_{\xi}$ on $\g{g}_{\gamma}^{\top} \oplus \phi_{\xi} (\g{g}_{\gamma}^{\top}) \oplus \phi_{\xi}^{2} (\g{g}_{\gamma}^{\top})$ consists of $\dim(\g{g}_{\gamma}^{\top})$ blocks of the form
	\[
	\frac{|\nu|}{\sqrt{2}} 
	\left(\begin{array}{ccc} 0& 1 & 0 \\   1 & 0 & 1  \\  0& 1 & 0  
	\end{array}\right).
	\]
	This shows that $\pm |\nu|$ are principal curvatures of $\Ss_{\xi}$ with 
multiplicities at least $\dim(\g{g}_{\gamma}^{\top})$. There are two other cases to analyze. Assume that $X \in \phi_{\xi}(V_{\gamma})$ and write $X = \phi_{\xi}(\eta)$ with $\eta \in V_{\gamma}$. From (\ref{definition:shape:operator}) and Proposition \ref{proposition:main2} we deduce
	\begin{align*}
	\Ss_{\xi} X & = \Ss_{\xi} \phi_{\xi} (\eta) = -\left(\nabla_{\phi_{\xi}(\eta)} \xi \right)^{\top} = \frac{|\nu|}{\sqrt{2}}(\phi^{2}_{\xi}(\eta) + \eta)^{\top} = \frac{|\nu|}{\sqrt{2}}\phi^{2}_{\xi}(\eta),\\
	\Ss_{\xi} \phi^{2}_{\xi}(\eta) & = - \left(\nabla_{\phi^{2}_{\xi} (\eta)} \xi\right) = \frac{|\nu|}{\sqrt{2}} \phi_{\xi} (\eta).
	\end{align*}
	So the $2$-dimensional vector space spanned by $ \phi_{\xi}(\eta), \phi_{\xi}^{2}(\eta) $ is $\Ss_{\xi}$-invariant for all $0 \neq \eta \in V_{\gamma}$ and all unit vectors $\xi \in V_{\nu}$. Thus the matrix representation of $\Ss_{\xi}$ on $\phi_{\xi}(V_{\gamma}) \oplus \phi_{\xi}^{2}(V_{\gamma})$ consists of $\dim(V_{\gamma})$ blocks of the form
	\[
	\frac{|\nu|}{\sqrt{2}} 
	\left(\begin{array}{cc} 0 & 1  \\   1 & 0  
	\end{array}\right).
	\]
	Consequently, $\pm |\nu|$ and $\pm \frac{|\nu|}{\sqrt{2}}$ are principal 
curvatures with multiplicities at least $\dim(\g{g}_{\gamma}^{\top})$ and 
$\dim(V_{\gamma})$, respectively. Finally, assume that $X \in \g{g}_{\gamma+\nu} \ominus \phi_{\xi}(\g{g}_{\gamma})$. From (\ref{definition:shape:operator}) and Proposition~\ref{proposition:main2} we deduce 
	\[
	\Ss_{\xi} X  = - (\nabla_{X} \xi)^{\top} = 0.
	\]
	This finishes the proof.
\end{proof}

We will now show that if $S \cdot o$ is a CPC submanifold, then all roots 
in $\psi$ must have the same length. We will start by investigating the symmetric spaces $G_{2}^{2}/SO_{4}$ and $G_{2}^{\mathbb{C}}/G_2$.

\begin{proposition}\label{proposition:G2}
	Let $M =G/K$ be a symmetric space of non-compact type whose Dynkin diagram is of type $G_2$. Let $\alpha_0$ and $\alpha_1$ be its simple roots. 
Let $S$ be the connected closed subgroup of $AN$ with Lie algebra $\g{s} = \g{s} \oplus (\g{n} \ominus V)$, where $V \subseteq \g{g}_{\alpha_0} \oplus \g{g}_{\alpha_1}$ has non-trivial projection onto $\g{g}_{\alpha_k}$ for $k \in \{0,1\}$. Then $S \cdot o$ cannot be a CPC submanifold of $M$. 
\end{proposition}

\begin{proof}
	We can assume $|\alpha_0| > |\alpha_1|$ and hence $|\alpha_0|^2 = 6$ and $|\alpha_1|^2 = 2$. The $\alpha_1$-string of $\alpha_0$ consists of $\alpha_0,\alpha_0 + \alpha_1,\alpha_0 + 2 \alpha_1,\alpha_0 + 3 \alpha_1$ and we have $A_{\alpha_1, \alpha_0} = -3$. Let $\xi_k \in V_{\alpha_k}$ be a unit vector and $k \in \{0,1\}$. We will determine a principal curvature of the shape operator $\Ss_{\xi_1}$ that cannot be a principal curvature of the shape operator $\Ss_{\xi_0}$. Note that $[\xi_1, \xi_0] \in \g{g}_{\alpha_0 + \alpha_1}$ is tangent to $S \cdot o$. Using (\ref{shape}) and Lemma~\ref{lemma:ad}(\ref{lemma:ad:ii}) we deduce 
	\begin{align*}
	2\Ss_{\xi_1} [\xi_1, \xi_0] & = -([[\xi_1, \xi_0], \xi_1] - [[\xi_1, \xi_0], \theta \xi_1])^{\top} \\
	& =   [\xi_1, [\xi_1, \xi_0]]^{\top} - [\theta \xi_1, [\xi_1, \xi_0]]^{\top}  \\
	& =  [\xi_1, [\xi_1, \xi_0]] - A_{\alpha_1, \alpha_0} |\alpha_1|^{2} \xi^{\top}_0 =   [\xi_1, [\xi_1, \xi_0]].
	\end{align*}
	Note that $A_{\alpha_1, \alpha_0 + \alpha_1} = -1$. From (\ref{shape}) 
and Lemma~\ref{lemma:ad}(\ref{lemma:ad:iii}), we obtain 
	\begin{align*}
	2\Ss_{\xi_1} [\xi_1, [\xi_1, \xi_0]]  
	&= - ([[\xi_1, [\xi_1, \xi_0]], \xi_1] - [[\xi_1, [\xi_1, \xi_0]], \theta \xi_1])^{\top}\\
	& =   [\xi_1, [\xi_1, [\xi_1, \xi_0]]]^{\top} -  [\theta \xi_1, [\xi_1, [\xi_1, \xi_0]]]^{\top}  
	\\& =  [\xi_1, [\xi_1, [\xi_1, \xi_0]]] + 8 [\xi_1, \xi_0].
	\end{align*}
	Finally, since $A_{\alpha_1, \alpha_0 + 2 \alpha_1} = 1$, from (\ref{shape}) and Lemma~\ref{lemma:ad}(\ref{lemma:ad:iv}) we conclude
	\begin{align*}
	2\Ss_{\xi_1} [\xi_1, [\xi_1, [\xi_1, \xi_0]]]  & = - [[\xi_1, [\xi_1, [\xi_1, \xi_0]]], \xi_1] + [[\xi_1, [\xi_1, [\xi_1, [\xi_1, \xi_0]]]], \theta \xi_1]^{\top}\\
	&  =  [\xi_1, [\xi_1, [\xi_1, [\xi_1, \xi_0]]]]^{\top} - [\theta \xi_1, [\xi_1, [\xi_1, [\xi_1, [\xi_1, \xi_0]]]]^{\top} \\&  = 6 [\xi_1, [\xi_1, \xi_0]]].
	\end{align*}
	Therefore, the $3$-dimensional vector space spanned by the three vectors 
$\ad(\xi_1)\xi_0$, $\ad^{2}(\xi_1)\xi_0$ and $\ad^{3}(\xi_1)\xi_0$ is $\Ss_{\xi_1}$-invariant. The corresponding matrix representation of $\Ss_{\xi_1}$ on that subspace is
	\[
	\left(
	\begin{array}{ccc}
	0 & 4 & 0 \\
	\frac{1}{2} & 0 & 3  \\
	0 & \frac{1}{2} & 0 \\
	\end{array}
	\right).
	\]
	The principal curvatures of $\Ss_{\xi_1}$ on this subspace are $\pm \sqrt{7/2}$ and $0$. If $S \cdot o$ is a CPC submanifold, then $\sqrt{7/2}$ must also be a principal curvature of the shape operator $\Ss_{\xi_0}$. However, since $|\alpha_0| > |\alpha_1|$, we deduce from Proposition~\ref{p2.48} that $|A_{\alpha_0, \mu}| \leq 1$ for all $\mu \in \Delta$. According to Proposition~\ref{proposition:principalcurvatures}(\ref{proposition:principalcurvatures:i}) all the non-trivial principal curvatures of $\Ss_{\xi_0}$ are $\pm \sqrt{3/2}$.  Therefore $S \cdot o$ cannot be a CPC submanifold.
\end{proof}

We now prove a similar result for symmetric spaces of non-compact type whose Dynkin diagram is not of type $G_2$.

\begin{proposition}\label{proposition:classification:lenght}
	Let $M = G/K$ be a symmetric space of non-compact type whose Dynkin diagram is not of type $G_2$. Let $S$ be the connected closed subgroup of $AN$ whose Lie algebra is $\g{s} = \g{s} \oplus (\g{n} \ominus V)$, where $V = \bigoplus_{\alpha \in \psi} V_{\alpha}$. If $S \cdot o$ is a CPC 
submanifold of $M$, then all roots in $\psi$ must have the same length.
\end{proposition}

\begin{proof}
	Assume that there are two roots $\alpha_0,\alpha_1 \in \psi$ with different length and that $|\alpha_0| > |\alpha_1|$. Then we have $|\alpha_0| = 
\sqrt{2} |\alpha_1|$ and there exists $\lambda \in \Delta^{+}$ with $A_{\alpha_1, \lambda} = -2$. Then $\lambda$ is the root of minimum level in 
its non-trivial $\alpha_1$-string, which consists of $\lambda,\lambda + \alpha_1,\lambda + 2\alpha_1$. Let $\xi_1 \in V_{\alpha_1}$ be a unit vector. Consider the restriction of the shape operator $\Ss_{\xi_1}$ to the tangent projection of the root spaces of the $\alpha_1$-string of $\lambda$. From Proposition~\ref{proposition:principalcurvatures}(\ref{proposition:principalcurvatures:ii}) we see that the non-zero principal curvatures of $\Ss_{\xi_1}$ are $\pm |\alpha_1|$, both with multiplicity $\dim(\g{g}_{\lambda}^{\top})$, and $\pm |\alpha_1|/\sqrt{2}$, both with multiplicity $\dim(V_{\lambda})$. In particular, the submanifold $S \cdot o$ is not totally geodesic. There exists  $\gamma \in \Delta^+$ such that its $\alpha_0$-string is non-trivial, because otherwise the shape operator $\Ss_{\xi_0}$ with respect to a unit vector $\xi_0 \in V_{\alpha_0}$ vanishes, which contradicts that $S \cdot o$ is a CPC submanifold. Without loss of generality we can assume that $\gamma$ is the root of minimum level in its 
$\alpha_0$-string. Since $\alpha_0$ is a long root, Proposition~\ref{p2.48} implies $A_{\alpha_0, \gamma} = -1$. From Proposition~\ref{proposition:principalcurvatures}(\ref{proposition:principalcurvatures:i}) we see that the non-zero principal curvatures of $\Ss_{\xi_0}$ are $\pm |\alpha_0|/2$, both with multiplicity $\dim(\g{g}_{\gamma}^{\top})$. But $|\alpha_1| \neq |\alpha_0|/2$. Since $S \cdot o$ is a CPC submanifold, it follows 
that $|\alpha_1|$ cannot be a principal curvature and hence $\dim(\g{g}_{\lambda}^{\top}) = 0$. In other words, $V_{\lambda} = \g{g}_{\lambda}$ and $\lambda$ is a simple root connected to $\alpha_1$ by a single edge. 
	
	We put $\alpha_2 = \lambda$ and define the normal vector $\xi = \cos(\varphi) \xi_1 + \sin(\varphi) \xi_2$, where $\xi_k \in V_{\alpha_k}$ for $k \in \{1,2\}$. Note that $\alpha_1,\alpha_2$ generate a root system of  type $B_2$ ($= C_2$). Therefore, according to (\ref{shape}) and (\ref{bracket:relation}), the vector space 
	\[
	\g{g}_{\alpha_1}^{\top} \oplus \g{g}_{\alpha_2}^{\top} \oplus \g{g}_{\alpha_1 + \alpha_2}^{\top}  \oplus  \g{g}_{2\alpha_1 + \alpha_2}^{\top} = 
 \g{g}_{\alpha_1}^{\top} \oplus \g{g}_{\alpha_1 + \alpha_2}  \oplus  \g{g}_{2\alpha_1 + \alpha_2}
	\]
	is $\Ss_{\xi}$-invariant. We will now investigate the shape operator $\Ss_{\xi}$ on this subspace. In fact, studying the principal curvatures of $\Ss_{\xi}$ when restricted to this subspace is equivalent to studying the principal curvatures of $S \cdot o$ as a submanifold of a rank $2$ symmetric space whose Dynkin diagram is of type $B_2$. First note that the $\alpha_2$-string containing $\alpha_1$ consists of $\alpha_1,\alpha_1 + \alpha_2$ and the $\alpha_1$-string containing $\alpha_2$ consists of $\alpha_2,\alpha_2 + \alpha_1,\alpha_2 + 2 \alpha_1$. We will use Proposition \ref{proposition:principalcurvatures} for both cases. On the one hand, the non-zero principal curvatures of $\Ss_{\xi_2}$ are $\pm |\alpha_2|/2$, both with multiplicity $\dim(\g{g}_{\alpha_1}^{\top})$. On the other hand, since $\g{g}_{\alpha_2} = V_{\alpha_2}$, the non-zero principal curvatures of $\Ss_{\xi_1}$ are $\pm |\alpha_1|/\sqrt{2}$, both with multiplicity $\dim(\g{g}_{\alpha_2}) = \dim(V_{\alpha_2})$. This implies that $\dim(\g{g}_{\alpha_2}) = \dim(\g{g}_{\alpha_1}^{\top})$ is a necessary condition for $S \cdot o$ to be a CPC submanifold. Since $V_{\alpha_1} \neq \{0\}$ by assumption, we get $\dim(\g{g}_{\alpha_1}) > \dim(\g{g}_{\alpha_2})$. This means, according to \cite[p.~337]{Jurgen}, that $S \cdot o$ 
must be contained in the symmetric space $SO^{o}_{r,r+n}/SO_{r} SO_{r+n}$, where $\dim(\g{g}_{\alpha_2}) = 1$ and $\dim(\g{g}_{\alpha_1}) = n$. Since $\dim(\g{g}_{\alpha_2}) = \dim(\g{g}_{\alpha_1}^{\top})$, $V_{\alpha_1}$ must be an $(n-1)$-dimensional subspace of $\g{g}_{\alpha_1}$. Let $\xi_2 \in V_{\alpha_2}$ and $X \in \g{g}_{\alpha_1}^{\top}$. From (\ref{definition:shape:operator}) and Proposition~\ref{proposition:main1} we deduce
	\[
	\Ss_{\xi_{2}} X = \frac{|\alpha_2|}{2} \phi_{\xi_2}(X) \  \mbox{  and  
} \ 
	\Ss_{\xi_{2}} \phi_{\xi_2}(X) = \frac{|\alpha_2|}{2} X.
	\]
	Now consider $\xi_1 \in V_{\alpha_1}$. Since $\phi^{2}_{\xi_1}  \rvert_{\g{g}_{\alpha_2}} \colon \g{g}_{\alpha_2} \to \g{g}_{2\alpha_1 +  \alpha_2}$ is a linear isometry, we have $\g{g}_{\alpha_2 + 2 \alpha_1} = \mathbb{R} \phi^{2}_{\xi_1}(\xi_2)$. Recall that $\phi_{\xi_2} \rvert_{\g{g}_{\alpha_1}} \colon \g{g}_{\alpha_1} \to \g{g}_{\alpha_1+ \alpha_2}$ is also linear isometry. Then, using (\ref{cartan:inner}) and combining definition (\ref{definition:phi}) together with Lemma~\ref{lemma:ad}(\ref{lemma:ad:ii}),(\ref{lemma:ad:i}), we obtain 
	\begin{align*}
	\langle (\phi_{\xi_1} \circ \phi_{\xi_2})(X), \phi^{2}_{\xi_1}(\xi_2)  \rangle_{AN} & = \langle \phi_{\xi_2}(X), \phi_{\xi_1}(\xi_2) \rangle_{AN} = - \langle \phi_{\xi_2}(X), \phi_{\xi_2}(\xi_1) \rangle_{AN} \\ &  = \langle X, \xi_1 \rangle_{AN} = 0.
	\end{align*}
	Therefore, using (\ref{definition:shape:operator}) and Proposition~\ref{proposition:main1}, we obtain  
	\[
	\Ss_{\xi_1} \phi_{\xi_2}(X) = \frac{|\alpha_1|}{\sqrt{2}}(\phi_{\xi_1} 
\circ \phi_{\xi_2} (X))^{\top} = 0.
	\]
	Since $\xi = \cos(\varphi) \xi_1 + \sin(\varphi) \xi_2$ with $\xi_k \in V_{\alpha_k}$ and $k \in \{1,2\}$, we deduce
	\[
	\Ss_{\xi}(X + \phi_{\xi_2}(X)) = \sin(\varphi) \frac{|\alpha_2|}{2}(X + \phi_{\xi_2}(X)),
	\]
	which shows that $S \cdot o$ cannot be a CPC submanifold.
\end{proof}

In order to finish this section, we just need to prove that $S \cdot o$ is not a CPC submanifold whenever there are at least two orthogonal roots in $\psi$. One of the consequences of \cite{BDT10} is that $S \cdot o$ is 
not a CPC submanifold when $\psi$ has exactly two orthogonal simple roots. The next result settles the general case. 

\begin{proposition}
	Let $\g{s} = \g{a} \oplus (\g{n} \ominus V)$ be a subalgebra of $\g{a} 
\oplus \g{n}$, for $V = \oplus_{\alpha \in \psi}\ V_{\alpha}$ and $\psi 
\subset \Pi'$. Assume that there are two orthogonal roots $\alpha_0,\alpha_1 \in \psi$. Let $S$ be the connected closed subgroup of $AN$ with Lie algebra $\g{s}$. Then the submanifold $S \cdot o$ is not a CPC submanifold.
\end{proposition}

\begin{proof}
	In view of Proposition~\ref{proposition:G2} and Proposition~\ref{proposition:classification:lenght} we can assume that all roots in $\psi$ have the same length. Taking into account the classification of Dynkin diagrams 
(see e.g.~\cite{K}), we deduce that there exist simple roots $\beta_1, \dots, \beta_r \in \Pi$ so that  $\alpha_0, \beta_1, \ldots, \beta_r, \alpha_1$ corresponds to a Dynkin diagram of type $A_{r+2}$. We define $\gamma 
= \sum_{i}^{r} \beta_i \in \Delta^+$. The $(\alpha_0, \alpha_1)$-string 
of $\gamma$ consists of $\gamma, \gamma + \alpha_0, \gamma + \alpha_1, \gamma + \alpha_0 + \alpha_1$. Let $\xi = \cos(\varphi) \xi_0 + \sin(\varphi) \xi_1$ be a unit normal vector with $\xi_k \in V_{\alpha_k}$ and $k \in \{0, 1\}$. Using (\ref{definition:shape:operator}) and Proposition \ref{proposition:main1}, we obtain that the non-trivial part of the matrix representation  $\Ss_{\xi}$ consists of $\dim(\g{g}_{\gamma}^{\top})$ blocks of the form
	\[
	\frac{|{\alpha_0}|}{2} 
	\left(\begin{array}{cccc} 0 & \cos(\varphi) & \sin(\varphi) & 0 \\  \cos(\varphi) & 0 & 0 & \sin(\varphi)  \\ \sin(\varphi) & 0 & 0 & \cos(\varphi)  \\ 0 & \sin(\varphi) & \cos(\varphi) & 0 
	\end{array}\right)
	\]
	with respect to $X, \phi_{\xi_0}(X), \phi_{\xi_1}(X), (\phi_{\xi_1} \circ \phi_{\xi_0})(X)$ for $X \in \g{g}_{\gamma}^{\top}$. The corresponding eigenvalues are $\pm \sqrt{1-\sin (2 \varphi )}$, both with multiplicity $2$. They clearly depend on $\varphi$, which cannot happen if $S \cdot o$ 
is a CPC submanifold. This implies $\g{g}_{\gamma} = V_{\gamma}$ and $\gamma=\beta_1 \in \Pi$. 
	
	Let $\xi_{\gamma} \in V_{\gamma}$ be a unit vector. Note that $\phi_{\xi_{\gamma}}(\xi_0) \in \g{g}_{\alpha_0 + \gamma}$ and that  $(\phi_{\xi_1} 
\circ \phi_{\xi_{\gamma}})(\xi_0) \in \g{g}_{\alpha_0 + \gamma +\alpha_1}$ are  tangent  to $S \cdot o$ at $o$. Using (\ref{definition:shape:operator}) and Proposition~\ref{proposition:main1}, we get  $2\Ss_{\xi_{\gamma}} \phi_{\xi_{\gamma}}(\xi_0) =  |\gamma| \xi^{\top}_0 = 0$ and 
	\[
	\Ss_{\xi_1} (\phi_{\xi_{\gamma}}(\xi_0) +(\phi_{\xi_1} \circ \phi_{\xi_{\gamma}})(\xi_0)) = \frac{|\alpha_1|}{2}(\phi_{\xi_{\gamma}}(\xi_0) +(\phi_{\xi_1} \circ \phi_{\xi_{\gamma}})(\xi_0)).
	\]
	Since $\alpha_0 + \alpha_1,\alpha_0 + 2 \gamma + \alpha_1 \notin \Delta$, using (\ref{shape}) together with (\ref{bracket:relation}) we deduce that $\Ss_{\xi_{\gamma}} (\phi_{\xi_1} \circ \phi_{\xi_{\gamma}})(\xi_0) = 
0$. Thus, if we define $\xi = \cos(\varphi) \xi_1 + \sin(\varphi)\xi_{\gamma}$, we get
	\[
	\Ss_{\xi} (\phi_{\xi_{\gamma}}(\xi_0) +(\phi_{\xi_1} \circ \phi_{\xi_{\gamma}})(\xi_0)) = \cos(\varphi) \frac{|\alpha_1|}{2}(\phi_{\xi_{\gamma}}(\xi_0) +(\phi_{\xi_1} \circ \phi_{\xi_{\gamma}})(\xi_0)).
	\]
	From this we see that $S \cdot o$ cannot be a CPC submanifold. This finishes the proof.
\end{proof}

\section{Description of the examples} \label{new:examples}

In this section we show that, with a few basic exceptions, the CPC submanifolds that we introduced in the Main Theorem are not singular orbits of cohomogeneity one actions.

Recall that $\alpha_0$ and $\alpha_1$ are two simple roots and $A_{\alpha_0, \alpha_1} = A_{\alpha_1, \alpha_0} = -1$. Recall also that $V$ is 
a subspace of $\g{g}_{\alpha_0} \oplus \g{g}_{\alpha_1}$ with non-trivial 
projections onto $\g{g}_{\alpha_0}$ and $\g{g}_{\alpha_1}$ (equivalently $V_0 \neq \{0\} \neq V_1$). We are studying the orbit $S \cdot o$, where $S$ is the connected closed subgroup of $AN$ with Lie algebra $\g{s} = \g{a} \oplus (\g{n} \ominus V)$. First, assume that  $V = \g{g}_{\alpha_0} \oplus \g{g}_{\alpha_1}$. Then $S \cdot o$ is one of the following submanifolds, or a canonical extension to $G/K$ of it:
\begin{itemize}
	\setlength{\itemsep}{1ex}
	\item[(i)] $\mathbb{R}H^2 \times \mathbb{R} \cong (SL_{2}(\mathbb{R})/SO_{2}) \times \mathbb{R}  \subset SL_{3}(\mathbb{R})/SO_{3}$,
	\item[(ii)] $\mathbb{R}H^3 \times \mathbb{R} \cong (SL_{2}(\mathbb{C})/SU_{2}) \times \mathbb{R}  \subset SL_{3}(\mathbb{C})/SU_{3}$,
	\item[(iii)] $\mathbb{R}H^5 \times \mathbb{R} \cong (SL_{2}(\mathbb{H})/Sp_{2}) \times \mathbb{R}   \subset SL_{3}(\mathbb{H})/Sp_{3}$,
	\item[(iv)]  $\mathbb{R}H^9 \times \mathbb{R} \subset E^{-26}_{6}/F_{4}$.
\end{itemize}
These four submanifolds  appear in the list \cite[Theorem 3.3]{BT04} of reflective submanifolds and are singular orbits of cohomogeneity one actions. Therefore, their canonical extensions are also singular orbits of cohomogeneity one actions. 

We will now see that the remaining submanifolds that we introduced in the 
Main Theorem do not admit such a description. One might study them in a rank 2 symmetric space and after that use some tools involving canonical extensions to conclude. However, for the sake of simplicity, we will carry 
out a direct study to avoid the introduction of these techniques.

Assume that $V_k$ is a proper subspace of $\g{g}_{\alpha_k}$ for $k \in \{0,1\}$ and that $\dim(\g{g}_{\alpha_0 + \alpha_1}) \geq 2$. We will assume that $S \cdot o$ is a singular orbit of a cohomogeneity one action and 
derive a contradiction. Up to now we used the Iwasawa decomposition to identify the tangent space $T_o(S \cdot o)$ of the orbit $S \cdot o$ at $o$ 
with $\g{s}$ and the normal space $\nu_o(S \cdot o)$ with $V$. However, in this section we will use the identification $\g{p} \cong T_o (G/K)$. This means that we will identify $T_o (S \cdot o)$ and $\nu_o (S \cdot o)$ with the orthogonal projections of $\g{s}$ and $V$ onto $\g{p}$, which are $(1- \theta) \g{s}$ and $(1- \theta) V$ respectively.

If $S \cdot o$ is the singular orbit of a cohomogeneity one action on $G/K$, then the normalizer $N_{K}(S \cdot o)$ of $S \cdot o$ in $K$ acts transitively on the unit sphere $\nu_o^1 (S \cdot o)$ in $\nu_o (S \cdot o)$. Let $\g{m}$ be the Lie algebra of $N_{K}(S \cdot o)$. Then we have $[\g{m}, \xi] = \nu_o (S \cdot o) \ominus \mathbb{R}\xi$ for each $\xi \in \nu_o^1 (S \cdot o)$. Let $\xi_0 \in V_0$ and $\xi_1 \in V_1$ be unit vectors. Taking into account that $\nu_o (S \cdot o) \cong (1- \theta) V$, 
there exists $Z \in \g{m}$ so that 
\begin{equation}\label{equation:transitive}
[Z, (1-\theta)\xi_0] = (1-\theta)\xi_1 \in \g{g}_{\alpha_1} \oplus \g{g}_{-\alpha_1}.
\end{equation}

Consider the orthogonal decomposition $\g{k} = \g{k}_0 \oplus \bigoplus_{\lambda \in \Delta^{+}} \g{k}_{\lambda}$ with $\g{k}_{\lambda} = \g{k} \cap (\g{g}_{\lambda} \oplus \g{g}_{-\lambda})$, and write $Z = Z_0 + 
\sum_{\lambda \in \Delta^{+}} Z_{\lambda}$ accordingly. On the one hand, we have 
\begin{equation}
\begin{aligned}\label{bracket:normalizer}
[Z_\lambda, (1-\theta) \xi_0] =(1-\theta) [Z_\lambda, \xi_0] \in \g{g}_{\lambda + \alpha_0} \oplus \g{g}_{-(\lambda + \alpha_0)} \oplus \g{g}_{ \lambda - \alpha_0}  \oplus \g{g}_{- (\lambda - \alpha_0)}
\end{aligned}
\end{equation}
for each $\lambda \in \Delta^{+}$. From (\ref{bracket:normalizer}) and (\ref{equation:transitive}), using $[\g{k}_0, \g{g}_{\lambda}] \subseteq \g{g}_{\lambda}$ for each $\lambda \in \Delta^{+}$, we deduce  $[Z_0, (1-\theta)\xi_0] = 0$. Thus, without loss of generality, we can assume that $Z_0 = 0$ and hence $Z =\sum_{\lambda \in \Delta^{+}} Z_{\lambda}$. From (\ref{bracket:normalizer}) and (\ref{equation:transitive}) we also see that $Z_{\alpha_0 + \alpha_1} \neq 0$. It is now easy to verify that
\[
N_{\g{k}}(T_o(S \cdot o)) = N_{\g{k}}(\nu_o(S \cdot o)) \subset \g{k}_0 
\oplus \g{k}_{\alpha_0 + \alpha_1} \oplus \left( \bigoplus_{\lambda \in \{\alpha_0, \alpha_1\}^{\perp}} \g{k}_{\lambda} \right),
\]
where $\{\alpha_0, \alpha_1\}^{\perp}$ denotes the set of positive roots that are orthogonal to both $\alpha_0$ and $\alpha_1$. Since $\g{m} \subset N_{\g{k}}(T_o(S \cdot o))$, we can thus write $Z =X + \theta X + \sum_{\lambda \in \{\alpha_0, \alpha_1\}^{\perp}} Z_{\lambda}$ with $0 \neq X \in \g{g}_{\alpha_0 + \alpha_1}$. Denote by $\g{l}$ the Lie algebra of $N_G (S \cdot o)$. It is clear that $\g{s} \subset \g{l}$ and $Z \in \g{l}$. Let $Y_1,\ldots,Y_q$ be an orthogonal basis of $\g{g}_{\alpha_0 + \alpha_1} \ominus \mathbb{R}X \subset \g{s}$, where $q = \dim(\g{g}_{\alpha_0 + \alpha_1})-1$. Note that the set $\{[Z, Y_i] = [\theta X,Y_i]: i = 1, \dots q \}$ generates a $q$-dimensional linear subspace $W$ of $\g{k}_0$, according Lemma~\ref{lemma:a,k0}(\ref{ak0}),(\ref{k0:elements}). Since $\g{l}$ is a subalgebra, we also have $W \subset \g{l}$ and therefore $W \subset N_{\g{k}}(T_o(S \cdot o))$. For $0 \neq \eta \in V_0$ we have
\[
[[Z, Y_i], (1-\theta) \eta]  = (1-\theta) [[Z, Y_i], \eta] = (1-\theta) [[\theta X, Y_i], \eta] \in (1-\theta)V,
\]
which is equivalent to $[[\theta X, Y_i], \eta] \in V_0$, $i \in \{1, \dots, q\}$. Note that we have  $[[\theta X, Y_i], \eta]=  [Y_i, \theta[\theta \eta, X]] \neq 0$ for all $i \in \{1, \dots, q \}$ by using twice Proposition~\ref{proposition:main1}, first for $[\theta \eta, X]$ and then for $[Y_i, \theta[\theta \eta, X]]$, taking into account that $\theta$ is 
an isomorphism of Lie algebras. Note also that $\langle [U, L], L \rangle_{B_{\theta}} = -\langle  L, [U,L] \rangle_{B_{\theta}}$ for all $U \in 
\g{k}_0$ and $L \in \g{n}$, which means that $[U, L]$ is orthogonal to $L$ for all $U \in \g{k}_0$ and $L \in \g{n}$. If $\dim(\g{g}_{\alpha_0 + \alpha_1}) =2$, then $V_0 = \mathbb{R} \eta$ is $1$-dimensional and $0 
\neq [[\eta, \theta X], Y_1]\in V_0$ is orthogonal to $\eta$, which is a contradiction. If $\dim(\g{g}_{\alpha_0 + \alpha_1}) > 2$, we have $0 \neq [[\theta X, Y_i], \eta] \in V_0$ for $i \in \{1, \dots q \}$. Since we have that $\dim(V_0) \leq \dim(T_0)$ by Proposition~\ref{characterisation} and Lemma~\ref{lemma:auxiliar:characterisation}(\ref{lemma:auxiliar:characterisation:iii}), these $q$ vectors must be linearly dependent. Thus
\[
0 = \sum_{i=1}^{q} a_i [[\theta X, Y_i], \eta] = \sum_{i=1}^{q} [[\theta X, a_i Y_i], \eta] =  [[\theta X, \sum_{i=1}^{q} a_i Y_i], \eta],
\]
which contradicts Proposition~\ref{proposition:main1} by the above argument. These contradictions come from the assumption that the action of $N_{K}(S \cdot o)$ on $\nu_o^1 (S \cdot o)$ is transitive. Therefore, if $V_k$ is a proper subset of $\g{g}_{\alpha_k}$ for $k \in \{0,1\}$, then the orbit $S \cdot o$ cannot be the singular orbit of a cohomogeneity one action.

\section{Further geometric explanations} \label{geometric explanations}

In this section we present a brief geometric context for some of the algebraic constructions in the previous sections.
Consider the inclusions
\[
SL_3(\R) \subset SL_3(\C) \subset SL_3(\HH) \subset E_6^{-26}.
\]
The maximal compact subgroup of $E_6^{-26}$ is $F_4$ and $E_6^{-26}/F_4$ is an exceptional Riemannian symmetric space of non-compact type whose root system is of type $A_2$. We have 
\[
SL_3(\R) \cap F_4 = SO_3\ ,\ SL_3(\C) \cap F_4 = SU_3\ ,\ SL_3(\HH) \cap F_4 = Sp_3.
\]
This leads to the totally geodesic embeddings
\[
SL_3(\R)/SO_3 \subset SL_3(\C)/SU_3 \subset SL_3(\HH)/Sp_3 \subset E_6^{-26}/F_4.
\]
The root system of these four Riemannian symmetric spaces $G/K$ is of type $A_2$ and the multiplicities of their roots are $1,2,4,8$, respectively. These dimensions correspond to the dimensions of the four normed real 
division algebras $\R,\C,\HH,\OO$. This suggests a close relation between 
these four symmetric spaces and normed real division algebras.

In fact, we have totally geodesic embeddings of the hyperbolic planes over these four normed real division algebras into these symmetric spaces:

\[
\begin{array}{ccccccc}
\R H^2 & \subset & \C H^2 & \subset & \HH H^2 & \subset & \OO H^2 \\[1ex]
\cap & & \cap & & \cap & & \cap \\[1ex]
SL_3(\R)/SO_3 & \subset & SL_3(\C)/SU_3 & \subset & SL_3(\HH)/Sp_3 & \subset & E_6^{-26}/F_4
\end{array}
\]
In each of the four cases, the totally geodesic submanifold $\FF H^2$ is reflective and hence there exists a totally geodesic  submanifold (which is also reflective) that is perpendicular to the hyperbolic plane. These are
\[
\begin{array}{ccccccc}
SL_3(\R)/SO_3 & \subset & SL_3(\C)/SU_3 & \subset & SL_3(\HH)/Sp_3 & \subset & E_6^{-26}/F_4 \\
\cup & & \cup & & \cup & & \cup \\
\R H^2 \times \R & \subset & \R H^3 \times \R & \subset & \R H^5 \times \R & \subset & \R H^9 \times \R \\
\perp & & \perp & & \perp & & \perp \\
\R H^2 & \subset & \C H^2 & \subset & \HH H^2 & \subset & \OO H^2 \\
\cap & & \cap & & \cap & & \cap \\
SL_3(\R)/SO_3 & \subset & SL_3(\C)/SU_3 & \subset & SL_3(\HH)/Sp_3 & \subset & E_6^{-26}/F_4
\end{array}
\]
The products $\R H^k \times \R$ are precisely our orbits $S \cdot o$ for the case when we remove $V = \g{g}_{\alpha_1} \oplus \g{g}_{\alpha_2}$. 
Thus the normal space $\nu_o(S \cdot o) \cong V$ of $S \cdot o$ at $o$ coincides with the tangent space $T_o\FF H^2$ of $\FF H^2$ at $o$ for a suitable $\FF H^2 \subset G/K$ and where $\FF$ is the corresponding division 
algebra. 

\medskip
Now suppose that $V$ is a proper subspace of $\g{g}_{\alpha_1} \oplus \g{g}_{\alpha_2} \cong T_o\FF H^2$.

If $\FF = \C$, then $V \cong \R \oplus \R \cong T_o\R H^2$ for a totally geodesic real hyperbolic plane $\R H^2 \subset \C H^2 \subset SL_3(\C)/SU_3$.

If $\FF = \HH$, then $V \cong \R \oplus \R \cong T_o\R H^2$ for a totally geodesic $\R H^2 \subset \HH H^2 \subset SL_3(\HH)/Sp_3$, or $V \cong \C \oplus \C \cong T_o\C H^2$ for a totally geodesic $\C H^2 \subset \HH H^2 \subset SL_3(\HH)/Sp_3$.

If $\FF = \OO$, then $V \cong \R \oplus \R \cong T_o\R H^2$ for a totally geodesic $\R H^2 \subset \OO H^2 \subset E_6^{-26}/F_4$, or $V \cong \C \oplus \C \cong T_o\C H^2$ for a totally geodesic complex hyperbolic plane $\C H^2 \subset \OO H^2 \subset E_6^{-26}/F_4$, or $V \cong \HH \oplus \HH \cong T_o\HH H^2$ for a totally geodesic $\HH H^2 \subset \OO H^2 \subset E_6^{-26}/F_4$.

In other words, this means that the totally geodesic hyperbolic planes in 
$G/K$ correspond to the subspaces $V$ that we can remove from $\g{g}_{\alpha_1} \oplus \g{g}_{\alpha_2}$ to obtain our CPC submanifolds.  

\medskip
The submanifolds $S \cdot o$ with $V$ strictly contained in $\g{g}_{\alpha_1} \oplus \g{g}_{\alpha_2}$ are some kind of ruled submanifolds. Here is a description for the simplest case when $G/K = SL_3(\C)/SU_3$ and $V 
\cong \R \oplus \R$. In this case we have the two reflective submanifolds 
$\R H^3 \times \R$ and $\C H^2$ which are perpendicular to each other at $o$. Consider the polar action on $\C H^2$ given in (ii)(d) in the Main Theorem of \cite{BD13}. The orbit of this polar action through $o$  is a Euclidean plane $\EE^2$, embedded in a horosphere of $\C H^2$ (equivalently, the $3$-dimensional Heisenberg group $N$) as a minimal surface. Perpendicular to $\EE^2$ at $o$ in $\C H^2$ is a totally geodesic $\R H^2 \subset \C H^2$. Moving this $\EE^2$ along $\R H^3 \times \R$ through the action on $\R H^3 \times \R$ by the solvable group $S'$ with $S' \cdot o = \R H^3 \times \R$ arising from the Iwasawa decomposition gives the orbit $S \cdot o$. Thus $S \cdot o$ is foliated by these Euclidean planes. The normal spaces are obtained by moving the totally geodesic $\R H^2$ perpendicular to $\EE^2$ in $\C H^2$ along $S \cdot o$. According to Proposition 3.4, the principal curvatures are $\pm 1/\sqrt{2}$  with multiplicity $1$ each and $0$ with multiplicity $4$. The $0$-eigenspace at $o$ is the tangent space at $o$ of the totally geodesic $\R H^3 \times \R$, and the other two eigenspaces arise from the non-totally geodesic minimal embedding of $\EE^2$.


\begin{thebibliography}{99}


	\bibitem{BB01}
J.~Berndt, M.~Br\"{u}ck:
Cohomogeneity one actions on hyperbolic spaces.
J.\ Reine.\ Angew.\ Math. \textbf{541} (2001), 209--235.

\bibitem{Jurgen} 
J.\ Berndt, S.\ Console, C.\ Olmos: 
Submanifolds and holonomy. Second edition. 
Monographs and Research Notes in Mathematics. CRC Press, Boca Raton, FL, 2016.

\bibitem{BD13} 
J.\ Berndt, J.C.\ D\'{\i}az-Ramos: 
Polar actions on the complex hyperbolic plane. 
Ann.\ Global Anal.\ Geom \textbf{43} (2013), no.\ 1, 99--106.

\bibitem{BDT10} 
J.\ Berndt, J.C.\ D\'{\i}az-Ramos, H.\ Tamaru: 
Hyperpolar homogeneous foliations on symmetric spaces of noncompact type. 
J.\ Differential Geom. \textbf{86} (2010), no.\ 2, 191--235.

%\bibitem{BD15} 
%J.\ Berndt, M.\ Dom\'inguez-V\'azquez: 
%Cohomogeneity one actions on some noncompact symmetric spaces of rank two. 
%\emph{Transform.\ Groups.} \textbf{20} (2015), no.\ 4, 921--938.

\bibitem{BT04} 
J.~Berndt, H.~Tamaru: 
Cohomogeneity one actions on noncompact symmetric spaces with a totally geodesic singular orbit. 
Tohoku \ Math.\ J.\ (2) \textbf{56} (2004), no.\ 2, 163--177.

\bibitem{BT07} 
J.~Berndt, H.~Tamaru: 
Cohomogeneity one actions on noncompact symmetric spaces of rank one. 
Trans.\ Amer.\ Math.\ Soc. \textbf{359} (2007), no.\ 7, 3425--3438.

\bibitem{BT13}
J.~Berndt, H.~Tamaru: 
Cohomogeneity one actions on symmetric spaces of noncompact type. 
J.\ Reine.\ Angew.\ Math. \textbf{683} (2013), 129--159.

\bibitem{Ca38}
\'{E}.~Cartan:
Familles de surfaces isoparam\'etriques dans les espaces \`{a} courbure constante. 
Ann.\ Mat.\ Pura Appl. \textbf{17} (1938), no.\ 1, 177--191.

\bibitem{Ch16}
Q.S.~Chi:
Isoparametric hypersurfaces with four principal curvatures, IV.
J.\ Differential Geom. \textbf{115} (2020), no.\ 2, 225--301.

\bibitem{Ch17}
Q.S.~Chi:
Classification of isoparametric hypersurfaces.
Proceedings of the Sixth International Congress of Chinese Mathematicians. Vol.\ I, 437--451, Adv.\ Lect.\ Math.\ (ALM ), 36, Int.\ Press, Somerville, MA, 2017.

\bibitem{DD13}
J.C.~D\'{\i}az-Ramos, M.~Dom\'{\i}nguez-V\'{a}zquez: 
Isoparametric hypersurfaces in Damek-Ricci spaces. 
Adv.\ Math. \textbf{239} (2013), 1--17.

\bibitem{DDS17} J.C.~D\'iaz-Ramos, M.\ Dom\'inguez-V\'azquez, V.~Sanmart\'in-L\'opez: 
Isoparametric hypersurfaces in complex hyperbolic spaces. 
Adv.\ Math. \textbf{314} (2017), 756--805.

\bibitem{miguel} M.\ Dom\'inguez-V\'azquez: 
Canonical extension of submanifolds and foliations in noncompact symmetric spaces. 
Int.\ Math.\ Res.\ Not.\ IMRN 2015, no.\ 22, 12114--12125.

%\bibitem{FKM81}
%D.~Ferus, H.~Karcher, H.F.~M\"{u}nzner: 
%Cliffordalgebren und neue isoparametrische Hyperfl\"{a}chen. 
%\emph{Math.\ Z.} \textbf{177} (1981), no.\ 4, 479--502.

\bibitem{GeTang} J.\ Ge, Z.\ Tang:
Geometry of isoparametric hypersurfaces in Riemannian manifolds. 
Asian J.\ Math. \textbf{18} (2014), no.\ 1,  117--125.

\bibitem{HL82}
R.~Harvey, H.B.~Lawson~Jr.:
Calibrated geometries.
Acta Math. \textbf{148} (1982), 47--157.

\bibitem{HOT91}
E.~Heintze, C.E.~Olmos, G.~Thorbergsson:
Submanifolds with constant principal curvatures and normal holonomy groups.
Internat.\ J.\ Math. \textbf{2} (1991), no.\ 2, 167--175.

\bibitem{K} A.W.~Knapp: 
Lie groups beyond an introduction. Second edition.
Progress in Mathematics, 140, Birkh\"auser Boston, Inc., Boston, MA, 2002.

\bibitem{Se38} B.~Segre: 
Famiglie di ipersuperficie isoparametriche negli spazi euclidei ad un qualunque numero di dimensioni. 
Atti Accad.\ Naz.\ Lincei Rend.\ Cl.\ Sci.\ Fis.\ Mat.\ Natur. (6) \textbf{27} (1938), 203--207.

\bibitem{Ta11} H.~Tamaru:
Parabolic subgroups of semisimple Lie groups and Einstein solvmanifolds. 
Math.\ Ann. \textbf{351} (2011), no.\ 1,  51--66.


\end{thebibliography}
\end{document}